\documentclass[a4paper,9pt]{amsart}
\usepackage{amstext, amsfonts,amsthm, amsmath, amssymb, amscd}
\usepackage[utf8]{inputenc}
\usepackage{calc}
\usepackage[colorlinks,linkcolor=blue,hyperindex]{hyperref}
\usepackage[justification=centering]{caption}
\usepackage{gensymb}
\usepackage{xfrac}  
\usepackage{xcolor}
\usepackage{multicol}
\usepackage{graphicx}
\usepackage{subfigure}
\usepackage{indentfirst}
\usepackage{caption} 
\usepackage{enumerate}
\usepackage[english]{babel}
\usepackage{fancyhdr}
\usepackage{refcount}
\usepackage{fullpage}
\usepackage{mathtools}
\usepackage{xparse} 
\usepackage{enumitem}
\usepackage{tabularx}
\usepackage{booktabs}
\usepackage{placeins}
\usepackage{tikz-cd}
\usepackage{todonotes}

\renewenvironment{abstract}
{\quotation\small\noindent\rule{\linewidth}{.5pt}\par\smallskip
	{\centering\bfseries\abstractname\par}\medskip}
{\par\noindent\rule{\linewidth}{.5pt}\endquotation}

\newtheorem{The}{Theorem}[section]
\newtheorem{Cor}[The]{Corollary}
\newtheorem{Def}[The]{Definition}
\newtheorem{Lem}[The]{Lemma}
\newtheorem{Rem}[The]{Remark}
\newtheorem{Exam}[The]{Example}
\newtheorem{Prop}[The]{Proposition}

\newtheorem*{Not}{Notation}

\newcommand{\re}{\ensuremath{\mathbb{R}}}
\newcommand{\co}{\ensuremath{\mathbb{C}}}
\newcommand{\fd}{\ensuremath{\mathbb{F}}}
\newcommand{\ha}{\ensuremath{\mathbb{H}}}
\newcommand{\bc}{\ensuremath{\mathbb{BC}}}
 
\newcommand{\nm}{\ensuremath{\Vert}} 
\newcommand{\nmr}{\ensuremath{\right\Vert}} 
\newcommand{\nml}{\ensuremath{\left\Vert}}  
\newcommand{\ebf}{\ensuremath{\textbf{e}}} 
\newcommand{\ebft}{\ensuremath{\textbf{e}^{\dagger}}}
\newcommand{\iu}{\mathrm{\bf i}\mkern1mu}
\newcommand{\ju}{\mathrm{\bf j}\mkern1mu}
\newcommand{\ku}{\mathrm{\bf k}\mkern1mu}

\DeclareMathOperator{\gl}{\ensuremath{GL}} 
\DeclareMathOperator{\rk}{\ensuremath{rank}} 
\DeclareMathOperator{\mn}{\ensuremath{M}} 
\DeclareMathOperator{\dete}{\ensuremath{det}} 
\DeclareMathOperator{\idt}{\ensuremath{Id}} 

\DeclareMathOperator{\ci}{\ensuremath{\mathbb{S}^{1}}}
\DeclareMathOperator{\cc}{\ensuremath{\mathbb{S}^{1}_{\co}}} 
\DeclareMathOperator{\tr}{\ensuremath{\mathbb{T}^{2}}} 
\DeclareMathOperator{\zd}{\ensuremath{\mathfrak{S}}}

\NewDocumentCommand{\mybar}{ O{0.7} O{2pt} m }{
    \mathrlap{\hspace{#2}\overline{\scalebox{#1}[1]{\phantom{\ensuremath{#3}}}}}\ensuremath{#3}
}
\newcommand*\cjt[1]{\widetilde{#1}}
\newcommand*\cjh[1]{\widehat{#1}}
\newcommand*\cj[1]{\mybar{#1}}
\renewcommand{\Re}{\operatorname{Re}}
\renewcommand{\Im}{\operatorname{Im}}

\makeatletter
\def\blfootnote{\gdef\@thefnmark{}\@footnotetext}
\makeatother

\title{Bicomplex Polar Weighted Homogeneous Polynomials}
\author{Yesenia Bravo}
\address{Instituto de Matemáticas, Unidad Cuernavaca, Universidad Nacional Autónoma de México, Avenida Universidad s/n, Colonia Lomas de Chamilpa, CP62210, Cuernavaca, Morelos Mexico}
\email{ybravoo.math@gmail.com}
\email{agustin.romano@im.unam.mx}
\author{Inácio Rabelo}
\address{Instituto de Ciências Matemáticas e de Computação, Av. Trabalhador São-Carlense 400, Centro. Caixa Postal: 668 CEP 13560-970, São Carlos SP, Brasil}
\email{rabeloinacio@usp.br}
\author{Agustín Romano-Velázquez}

\numberwithin{equation}{section}

\begin{document}

\blfootnote{\textup{2020} \textit{Mathematics Subject Classification.} Primary: 32S55, 30G35; Secondary: 32C18, 14B05}

\maketitle

\begin{abstract}
   We study the topology of real polynomial maps $\re^{4n} \longrightarrow \re^{4}$ expressed in terms of bicomplex variables and their conjugates, which we refer to as bicomplex mixed polynomials. We introduce the notion of polar weighted homogeneity, a property that generalizes the concept of weighted homogeneity in the complex setting. This leads to the existence of global and spherical Milnor fibrations. Moreover, we include a discussion on bicomplex vector calculus, a bicomplex holomorphic analogue of the Milnor fibration theorem, and a theorem of Join type that describes the homotopy type of the fibers of certain polynomials on separable variables. This extends previous works on mixed polynomials in complex variables and their conjugates.
\end{abstract}

\section*{Introduction}

The bicomplex numbers, denoted by $\bc$, form a 4-dimensional commutative real algebra with zero divisors that generalizes the field of complex numbers. This algebra inherits a complex structure and possesses several properties that, in some contexts, make it more advantageous than the quaternions. Its foundational aspects and the theory of bicomplex holomorphic functions were first studied by the Italian school of Segre in a series of papers beginning with \cite{Segre1892}. Other fundamental contributions are found in \cite{Riley1953, Ryan1982}, and in the book \cite{Price1991}. More recent interest in bicomplex structures and their applications is reflected in works such as \cite{Colombo2014, Colombo2011, Resendis2023, Rochon2004a, Rochon2024b}, and \cite{Bravo2023}, to name a few. A modern treatment is provided in the book \cite{Luna2015}.

On the other hand, Milnor's fibration theorem, introduced in \cite{Milnor1968}, is a landmark result in Singularity Theory that describes the topology of complex varieties near their critical points. Analogous statements for real analytic maps hold under stronger conditions. The first examples, given in \cite{Ruas2002}, of real analytic maps that satisfy an analogue of Milnor's fibration theorem were real polynomial maps $\co^{n} \longrightarrow \co$ written on complex variables and their conjugates, now known as \textit{mixed polynomials}. These objects generalize their complex counterpart and play a central role in investigating the topology of singularities. For more details, see \cite{Oka2021} and the references therein. 

In \cite{Cisneros2008}, a notion of polar weighted homogeneity, related to a $\co^{*}$-action on $\co^{n}$, was introduced for mixed polynomials based on the works \cite{Seade1997} and \cite{Ruas2002}. This leads to the existence of a global fibration and a Milnor fibration on the sphere, both defined on the complements of the zero set. Moreover, Join type theorem was proved following \cite{Oka1973}. This result describes the homotopy type of the fiber of polar weighted homogeneous polynomials that decompose as a sum of polynomials on separable variables. 

The main goal of this work is to study the topology of real polynomial maps $\bc^{n} \longrightarrow \bc$, regarded as \textit{bicomplex mixed polynomials}. For this, we generalize the results of complex and real analytic maps that we discussed previously. Using the trigonometric representation in bicomplex algebra, we define a $\bc^{*}$-action on $\bc^{n}$, where $\bc^{*}$ denotes the non-zero divisors, and introduce a similar notion of polar weighted homogeneity. As a consequence, we derive the existence of fibrations results as before on the complement of the preimage of zero divisors. In the spherical case, we obtain two fibrations, whose base spaces are an open connected subset of $\mathbb{S}^{3}$ and a complex quadric. Furthermore, a Join type theorem follows with minor modifications in the original proof.

We include a discussion on bicomplex vector calculus. While the definitions resemble those in the complex case, but it can be developed independently, as we shall see. In addition, we discuss a Milnor-type fibration theorem for bicomplex holomorphic maps and explain how the classical theory applies to the underlying real structures. The key idea of our work is that the language of bicomplex variables provides a natural and powerful framework for studying real polynomial maps $\mathbb{R}^{4n} \to \mathbb{R}^4$. Although real algebras have appeared before in Singularity Theory, their role in the context of fibrations has only recently been considered, as discussed in \cite{Massey2008}.

\section{Bicomplex algebra}
We refer the reader to \cite{Luna2015} for more details about the results of this section. The set of bicomplex numbers is defined by
\begin{align*}
    \bc = \{ \lambda_{1} + \ju \lambda_{2} : \lambda_{1}, \lambda_{2} \in \co(\iu)\},
\end{align*}
where $\co(\iu) = \{ a + \iu b: a,b \in \re\}$ denotes the usual complex numbers and $\ju$ is a second imaginary unit satisfying $\iu \neq \ju$ and $\iu\ju = \ju\iu$. From now on, by $\co$ we mean the set $\co(\iu)$. A \textit{bicomplex number} $Z = \lambda_{1} +\ju \lambda_{2} \in \bc$ can be identified with the pair $(\lambda_{1}, \lambda_{2})$ where $\lambda_{1}$ and $\lambda_{2}$ are complex numbers. Thus, $\bc \simeq \co^{2} \simeq \re^{4}\simeq \ha$ as real vector spaces, where $\ha$ denotes the quaternions. Therefore, $\bc$ is a 4-dimensional commutative real algebra and a basis is given by $\langle 1, \iu, \ju, \ku \rangle$, where $\ku = \iu \ju$. Beyond commutativity, the main difference with quaternions is the existence of zero divisors. Namely, define the elements
\begin{align*}
    \ebf = \frac{1 + \iu\ju}{2}, \;\;\; \ebft = \frac{1 - \iu\ju}{2}.
\end{align*}
The numbers $\ebf$ and $\ebft$ have the following properties:
\begin{align*}
    &\ebf \cdot \ebft = \ebft\cdot\ebf = 0; \\
    &\ebf^{2} = \ebf, \;\;\left(\ebft\right)^{2} = \ebft; \\
    &\ebf + \ebft = 1, \;\; \ebf - \ebft = \iu\ju.
\end{align*}
The set $\langle \ebf, \ebft \rangle$ is a complex basis for $\bc$. Thus, every $Z = \lambda_{1} + \ju \lambda_{2} \in \bc$ can be uniquely written as:
\begin{align*}
    Z = z_{1}\ebf + z_{2}\ebft,
\end{align*}
where $z_{1} = \lambda_{1}-\iu \lambda_{2}$ and $z_{2} = \lambda_{1} + \iu \lambda_{2}$. This is called the \textit{idempotent representation} of $Z$. The elementary operations in $\bc$ can be easily performed component by component on this basis. It is important to remark that the existence of the elements $\ebf$ and $\ebft$ with the above properties is quite special and has no analogue in $\co$. Moreover, the set of invertible elements is described by
\begin{align*}
    \bc^{*} := \left\{ Z \in \bc : Z = z_{1}\ebf + z_{2}\ebft, \;z_{1},z_{2} \neq 0\right\}.
\end{align*}
We denote its complement, the set of zero divisors, by $\zd$. It follows that $\zd$ is the union of two planes in $\bc \simeq \re^{4}$.

Another difference with the complex numbers is that the bicomplex numbers admit three types of conjugations. Let $Z = \lambda_{1} + \ju \lambda_{2}$ and write its idempotent representation as $z_{1}\ebf + z_{2}\ebft$. We denote by:
\begin{align}
    \cjt{Z} &= \bar{\lambda}_{1} - \ju\bar{\lambda}_{2} = (\bar{\lambda}_{1}+\iu\bar{\lambda}_{2})\ebf + (\bar{\lambda}_{1}-\iu\bar{\lambda}_{2}) \ebft.    \label{cjt}\\
    \cjh{Z} &= \lambda_{1} - \ju \lambda_{2} = (\lambda_{1}+\iu \lambda_{2})\ebf + (\lambda_{1}-\iu \lambda_{2})\ebft. \label{cjh} \\
    \cj{Z} &= \bar{\lambda}_{1} +\ju\bar{\lambda}_{2} = (\bar{\lambda}_{1}-\iu\bar{\lambda}_{2})\ebf + (\bar{\lambda}_{1}+\iu\bar{\lambda}_{2})\ebft. \label{cj}
\end{align}
In addition, these conjugations satisfy the following relations:
\begin{equation}\label{eqc2}
    \begin{aligned}
        \lambda_{1} = \frac{Z + \cjh{Z}}{2}, \;\;\; \lambda_{2} = \frac{Z- \cjh{Z}}{2\ju}, \\
        \bar{\lambda}_{1} = \frac{\cjt{Z}+ \cj{Z}}{2}, \;\;\; \bar{\lambda}_{2} = \frac{\cjt{Z}+ \cj{Z}}{2\ju}.
    \end{aligned}
\end{equation}
It should be noted that expressing an arbitrary vector in $\re^{4}$ using bicomplex variables requires all three conjugations.

\subsection{Trigonometric form}
The existence of the three conjugations produces that any non-zero-divisor of $\bc$ has two polar (or trigonometric) representation. Moreover, we have the usual Euclidean norm and a complex-valued norm: 
\begin{align*}
 \nm Z \nm&= \nm \lambda_{1} + \ju \lambda_{2}\nm := \sqrt{ |\lambda_{1}|^{2} + |\lambda_{2}|^{2}},\\
    \nm Z \nm_{\iu} &= \nm \lambda_{1} + \ju \lambda_{2}\nm_{\iu} := \sqrt{ \lambda_{1}^{2} + \lambda_{2}^{2}},
\end{align*}
with the following convention: if $\lambda_{1}^{2}+\lambda_{2}^{2}$ is a non-negative real number, the square root is the usual, otherwise, we take the solution in which the imaginary part is positive. The set of complex numbers with positive imaginary parts is referred in the literature by \textit{upper half-plane} and we shall denote its union with the positive horizontal axis by $\mathcal{H}^{+}$.A straightforward computation shows that $Z \in \bc^{*}$ if and only if $\nm Z \nm_{\iu} \in \co^{*}$. This motivates us to consider the following set, called the \textit{complex unit circle}:
\begin{align*}
    \cc = \{ (z,w) \in \co^{2} : z^{2} + w^{2} = 1\}.
\end{align*}
If $Z = \lambda_{1} + \ju \lambda_{2} \in \bc^{*}$, then $\nm Z \nm_{\iu} \neq 0$ and we may write
\begin{align*}
    \left(\frac{\lambda_{1}}{\nm Z \nm_{\iu}}\right)^{2} + \left(\frac{\lambda_{2}}{\nm Z \nm_{\iu}}\right)^{2} = 1.
\end{align*}
We set
\begin{align*}
    \cos \Theta = \frac{\lambda_{1}}{\nm Z \nm_{\iu}}, \;\;\; \sin \Theta = \frac{\lambda_{2}}{\nm Z \nm_{\iu}},
\end{align*}
for some complex angle $\Theta$. By the periodicity of the complex sine and cosine, there are infinitely many solutions for the system above, and we shall call any of them the \textit{complex argument} of $Z$. Restricting $\Re(\Theta) \in [0,2\pi)$, this solution is called the \textit{principal value} and denoted by $\arg_{\iu}Z$. Thus, we may write
\begin{align*}
    Z = \nm Z \nm_{\iu} e^{\ju\arg_{\iu}Z}, \quad e^{\ju\arg_{\iu}Z} \in \cc.
\end{align*}
This yields a polar coordinate system $(\omega, \Theta)$ for $\bc^{*}$ given by
\begin{align*}
    \mathcal{H}^{+} \times \cc &\to \bc^{*}\\
    (\omega, \Theta) &\mapsto \omega \left( \cos{\Theta}+\ju\sin{\Theta}\right)
\end{align*}
Associated with this representation, we have the projection $\pi_{\iu} : \bc^{*} \longrightarrow \cc$ given by
\begin{align}\label{prj1}
    \pi_{\iu}(Z) = \frac{Z}{\nm Z \nm_{\iu}}.
\end{align}
This map admits an interesting geometric interpretation. Let us denote $\mathbb{S}^{3}_{0} := \mathbb{S}^{3} \setminus \zd$ the subset of the usual Euclidean sphere consisting of non-zero divisors. Notice that
\begin{align*}
    \mathbb{S}^{3}_{0} = \{(\omega, \Theta) \in \bc^{*} : \nm \omega \Theta \nm = 1\}. 
\end{align*}
\begin{Prop}
    The projection $\pi_{\iu} : \mathbb{S}^{3}_{0} \longrightarrow \cc$ is equivalent to a restriction of the Hopf fibration of $\mathbb{S}^{3}$.
\end{Prop}
    \begin{proof}
        The fibres $\pi_{\iu}^{-1}(\Theta_{0})$ are copies of $\mathbb{S}^{1} \cap \mathcal{H}^{+}$ and are described by
        \begin{align*}
            \pi_{\iu}^{-1}(\Theta_{0}) = \left\{ (\omega, \Theta_{0}) : \nm \omega \nm = \nm e^{-\ju\Theta_{0}}\nm\right\}.
        \end{align*}
        Therefore, every invertible $Z \in \mathbb{S}^{3}_{0}$ belongs to a fibre $\pi_{\iu}^{-1}(\Theta_{0})$. Hence,
        \begin{equation*}
            \mathbb{S}^{3}_{0} =\bigsqcup_{\Theta \in \cc} \pi_{\iu}^{-1}(\Theta_{0}).
        \end{equation*}
        Now, the proof follows from the main theorem in~ \cite{Gluck}.
    \end{proof}

\begin{Rem}
    \normalfont A routine exercise shows that $\cc$ is diffeomorphic to $T\mathbb{S}^{1}$, the tangent bundle of $\mathbb{S}^{1}$, which is equivalent to the 2-sphere minus two points. Then the previous assertion says that the projection $\pi_{\iu}$ is the result of removing two circles corresponding to the intersection $\mathbb{S}^{3} \cap \zd$ and that are mapped to two points in $\mathbb{S}^{2}$.
\end{Rem}

\subsection{Hyperbolic trigonometric form}
Let $Z = z_{1}\ebf + z_{2}\ebft \in \bc^{*}$ where $z_{1} = r_{1}e^{\iu\theta_{1}}$ and $z_{2} = r_{2}e^{\iu\theta_{2}}$. The \textit{hyperbolic polar form} of $Z$ is given by
\begin{align*}
    Z = \left( r_{1}\ebf + r_{2}\ebft\right)\left(e^{\iu\theta_{1}}\ebf + e^{\iu\theta_{2}}\ebft\right).
\end{align*}
The set
\begin{align*}
    \mathbb{D}^{+} = \{ \nu\ebf + \mu\ebf^{\dag} : \nu,\mu \in \re^{+}\}.
\end{align*}
is called the \textit{positive hyperbolic numbers}.
Let $Z = z_{1}\ebf + z_{2}\ebft$ and define
\begin{equation*}
 \nm Z \nm_{\ku} = \nm z_{1} \nm \ebf + \nm z_{2} \nm\ebft.   
\end{equation*}
The set $\mathbb{D}^{+}$ admits a partial order as follows: If $W = w_{1}\ebf + w_{2}\ebft$, we say that $\nm Z \nm_{\ku} \preceq \nm W \nm_{\ku}$ if and only if $\nm z_{1} \nm \le \nm w_{1} \nm$ and $\nm z_{2} \nm \le \nm w_{2} \nm$. The following properties are immediate:
\begin{enumerate}
    \item $\nm Z \nm_{\ku} = 0$ if and only if $Z = 0$.
    \item $\nm Z\cdot W \nm_{\ku} = \nm Z \nm_{\ku} \cdot \nm W \nm_{\ku}$.
    \item $\nm Z + W \nm_{\ku} \preceq \nm Z\nm_{\ku} + \nm W \nm_{\ku}$.
\end{enumerate}
Hence, each $Z \in \bc^{*}$ can be uniquely written as $Z = \nm Z \nm_{\ku}\lambda$, where $\nm Z \nm_{\ku} \in \mathbb{D}^{+}$ and $\lambda \in \ci\ebf \times \ci\ebft$. We shall denote $\tr = \mathbb{S}^{1}\ebf + \mathbb{S}^{1}\ebft$. As for complex numbers, there are infinitely many solutions to the equation $\theta = \theta_{1}\ebf + \theta_{2}\ebft \in \tr$. Then we restrict, $\{\theta_{1}, \theta_{2}\} \in [0,2\pi)$ and such a solution is called the \textit{hyperbolic argument} and denoted by $\arg_{\mathbb{D}}Z$. Associated with this representation, we have the projection $\pi_{\ku} : \bc^{*} \longrightarrow \tr$ given by
\begin{align}\label{prj2}
    \pi_{\ku}(Z) = \frac{z_{1}}{\nm z_{1}\nm}\ebf + \frac{z_{2}}{\nm z_{2}\nm}\ebft,
\end{align}
which is the normalization of each coordinate of the idempotent representation of $Z$.

\subsection{Linear algebra}
As a commutative ring, the theory of bicomplex modules can be developed as usual, and the idempotent representation of these numbers allows us to give a clear description of the $\bc$-linear transformations. First, let $A = (Z_{i,j})_{m\times n}$ be a bicomplex matrix and decompose $Z_{i,j} = z^{1}_{i,j}\ebf + z^{2}_{i,j}\ebf^{\dag}$ for each pair $(i,j)$, where $z^{1}_{i,j}, z^{2}_{i,j} \in \co$. It follows that $A$ has a unique decomposition $A = A^{1}\ebf + A^{2}\ebf^{\dag}$, where $A^{1} = (z^{1}_{i,j})_{m\times n}$ and $A^{2} = (z^{2}_{i,j})_{m\times n}$ are complex matrices. Let us denote the matrices of order $m \times n$ with coefficients in an algebra $\fd$ by $\mn(m\times n,\fd)$. Then the following map defines an embedding $\mn(m\times n, \bc) \hookrightarrow \mn(2m \times 2n, \co)$:
\begin{align}\label{embb}
    A = A^{1}\ebf + A^{2}\ebf^{\dag} \longmapsto \begin{pmatrix}
        A^{1} & 0 \\
        0 & A^{2}
    \end{pmatrix}.
\end{align}
The $\mn(m\times n, \bc)$-action on $\bc^{n}$ is thus described by a $\mn(2m\times 2n, \co)$-action on $\bc^{n} \simeq \co^{n} \times \co^{n}$.
\begin{Def}Let $A = A^{1}\ebf + A_{2}\ebft \in \mn(m\times n, \bc)$ be a bicomplex matrix.
    \begin{enumerate}
        \item The bicomplex rank of $A$, denoted by $\rk(A)$, is the pair $\left( \rk(A^{1}), \rk(A^{2})\right)$, where $\rk (A^{i})$ is the usual complex rank of $A^{i}$, for $i=1,2$.
        \item The determinant of $A$ is defined by
        \begin{align*}
            \dete(A) = \dete(A^{1})\ebf + \dete(A^{2})\ebf^{\dag} \in \bc,
        \end{align*}
        where $\dete(A^{i})$ is the complex determinant of $A^{i}$, for $i=1,2$.
    \end{enumerate}
\end{Def}
Since the sum and product of a bicomplex scalar preserves the idempotent representation, the determinant, as defined above, is multilinear, and $\dete(I_{n}) = 1$, where $I_{n}$ is the identity matrix. By the uniqueness of a function with this property, the function above is the legitimate bicomplex determinant. Moreover, notice that $A \in \mn(n,\bc)$ is invertible if and only if $A^{1}$ and $A^{2}$ are also invertible. This implies that $A$ is invertible if and only if $\dete(A) \notin \zd$. Therefore, the set
\begin{align*}
    \gl(n,\bc) = \{ A \in \mn(n,\bc) : \dete(A) \notin \zd \},
\end{align*}
form a group and the map \eqref{embb} defines an embedding $\gl(n,\bc) \hookrightarrow \gl(2n,\co)$.

\section{Bicomplex vector calculus}
According to the previous subsection, we consider the following hyperbolic norm on $\bc^{n}$: Let $Z \in \bc^{n}$ and write $Z = z_{1}\ebf + z_{2}\ebft$, where $z_{1}, z_{2} \in \co^{n}$. We define
\begin{align*}
    \nm Z \nm_{\ku} := \nm z_{1} \nm\ebf + \nm z_{2} \nm\ebft \in \mathbb{D}^{+}.
\end{align*}
In this section, we introduce the basic theory of bicomplex vector calculus in analogy with the calculus of several variables. The proofs are easy to deduce from these classical results (see, for instance, \cite{Spivak1968} and \cite{Gunning2009}). In this section, $\mathcal{U} = \mathcal{U}_{1}\ebf + \mathcal{U}_{2}\ebft \subset \bc^{n}$ denotes an open subset, where $\mathcal{U}_{1}$ and $ \mathcal{U}_{2}$ are open subsets of $\mathbb{C}^n$.

\begin{Def}\label{drv}
    A bicomplex map $F: \mathcal{U} \subset \bc^{n} \longrightarrow \bc^{m}$ is bicomplex differentiable at the point $Z \in \mathcal{U}$ if there exists a bicomplex linear transformation $L_{Z} : \bc^{n} \longrightarrow \bc^{m}$ such that
    \begin{align}\label{diflimit}
        \lim_{\substack{H \to 0 \\ \nm H \nm_{\ku} \not \in \zd}} \frac{ \nm F(Z+H) - F(Z) - L_{Z}(H)\nm_{\ku}}{\nm H \nm_{\ku}} = 0.
    \end{align}
    The transformation $L_{Z}$ is denoted by $DF_{Z}$ and it is called the bicomplex derivative of $F$ at $Z$. If $F$ is differentiable at all points $Z \in \mathcal{U}$, we say that $F$ is bicomplex holomorphic.
\end{Def}

Notice that $\nm H \nm_{\ku} \to 0$ if and only if $H \to 0$. Thus, the definition above coincides with \cite[Definition 7.2.1]{Luna2015} and \cite[Theorem 3.3]{Resendis2022} for $n = m =1$. We shall see that the hyperbolic norm allows us to apply the classical arguments of the calculus of several variables.

\begin{Prop}
    Let $F: \mathcal{U} \subset \bc^{n} \longrightarrow \bc^{m}$ be a bicomplex differentiable map at $Z \in \mathcal{U}$. Then, the linear transformation $L_{Z}$ is unique.
\end{Prop}
    \begin{proof}
        Suppose the existence of a second linear transformation $M_{Z}$ with the above properties. One has 
        {\small
        \begin{align*}
            \lim_{\substack{H \to 0 \\ \nm H \nm_{\ku} \not \in \zd}} \frac{ \nm L_{Z}(H) - M_{Z}(H)\nm_{\ku}}{\nm H\nm_{\ku}} &= \lim_{\substack{H \to 0 \\ \nm H \nm_{\ku} \not \in \zd}} \frac{\nm L_{Z}(H) - F(Z+H) + F(H) + F(Z+H) - F(H) - M_{Z}(H)\nm_{\ku}}{\nm H \nm_{\ku}} \\
            &\preceq \lim_{\substack{H \to 0 \\ \nm H \nm_{\ku} \not \in \zd}} \frac{\nm F(Z+H) - F(Z) - L_{Z}(H)\nm_{\ku}}{\nm H \nm_{\ku}} + \lim_{\substack{H \to 0 \\ \nm H \nm_{\ku} \not \in \zd}} \frac{\nm F(Z+H) - F(Z) - M_{Z}(H) \nm_{\ku}}{\nm H \nm_{\ku}} \\
            &= 0.
        \end{align*}}
        Let $W \in \bc^{n}$ and notice that $\nm \lambda W \nm_{\ku} \rightarrow 0$ if $\lambda \rightarrow 0$. Thus, if $\nm W \nm_{\ku} \neq 0$, we have that
        \begin{align*}
            0 = \lim_{\substack{\lambda \to 0 \\ \nm \lambda \nm_{\ku} \notin \zd}}\frac{\nm L_{Z}(\lambda W) - M_{Z}(\lambda W)\nm}{\nm(\lambda W)\nm_{\ku}} = \frac{\nm L_{Z}(W)- M_{Z}(W)\nm_{\ku}}{\nm W \nm_{\ku}},
        \end{align*}
        which implies $L_{Z}(W) = M_{Z}(W)$ for all $W \in \bc^{n}$ such that $\nm W \nm_{\ku} \neq 0$.
    \end{proof}
\begin{Def}
    Let $F : \mathcal{U} \subset \bc^{n} \longrightarrow \bc$ be a bicomplex function. The $i$-th partial derivative of $F$ at a point $W = (W_{1}, \dots, W_{n})$ is the following limit, if it exists,
    \begin{align*}
        \frac{\partial F}{\partial Z_{i}}(W) := \lim_{\substack{H \to 0 \\ H \notin \zd}} \frac{F(W_{1}, \dots, W_{i-1}, W_{i}+H, W_{i+1}, \dots, W_{n}) - F(W)}{\nm H \nm_{\ku}}.
    \end{align*}
\end{Def}
\begin{The}\label{form}
    Let $F = (F^{1}, \dots, F^{m}) : \mathcal{U} \subset \bc^{n} \longrightarrow \bc^{m}$ be a bicomplex map where each coordinate map is also a bicomplex map.
    \begin{enumerate}
        \item The map $F$ is bicomplex holomorphic at $W \in \mathcal{U}$ if and only if $F^{i}$ is bicomplex holomorphic at $W$ for every $i=1,\dots, n$.
        \item If $F$ is bicomplex holomorphic, then $DF_{W} = \left(DF^{1}_{W}, \dots, DF^{m}_{W}\right)$.
        \item If $F$ is bicomplex holomorphic, then the partial derivatives of $F^{i}$ at $W$ exist for every $i=1, \dots, n$ and 
        \begin{align*}
            DF_{W} = \left( \frac{\partial F^{i}}{\partial Z_{j}}(W)\right)_{m\times n}
        \end{align*}
    \end{enumerate}
\end{The}
    \begin{proof}
        For each $i$, the $i$-th entry of the bicomplex vector $F(W+H)-F(W)-L_{W}(H)$ is $F^{i}(W+H) - F^{i}(W) - L_{W}^{i}(H)$, where $L^{i}_{W}$ is the $i$-th row of $L_{W}$ and $H \in \bc^{n}$ is such that $H \to 0$ and $\nm H \nm_{\ku} \notin \zd$. One has
        {\small
        \begin{align*}
            \frac{\nm F^{i}(W+H) - F^{i}(W) - L_{W}^{i}(H) \nm_{\ku}}{\nm H \nm_{\ku}} \preceq \frac{\nm F(W+H)-F(W)-L_{W}(H) \nm_{\ku}}{\nm H \nm_{\ku}} \preceq \sum_{j=1}^{m} \frac{\nm F^{j}(W+H) - F^{j}(W) - L_{W}^{j}(H) \nm_{\ku}}{\nm H \nm_{\ku}}.
        \end{align*}}
        \hskip-4pt A straightforward argument shows that these inequalities imply the first and the second assertions. In virtue of that, for the third item, we only need to consider the case $m=1$. In Definition \ref{drv} we take $H = (0, \dots, 0, H_{i}, 0, \dots, 0)$ with $H_{i} \to 0$ and $H_{i} \notin \zd$. This leads to the existence of the partial derivatives of every $F^{i}$ and that $DF^{i}_{W}(E^{j}) = \frac{\partial F^{i}} {\partial Z_{j}}$, where $E^{j} = (0, \dots, 0, 1, 0, \dots, 0) \in \bc^{n}$ is the $j$-th canonical bicomplex vector. But this completely determines $DF^{i}_{W}$ and we conclude.
    \end{proof}

\begin{Cor}\label{bicor1} Let $F : \mathcal{U} \subset \bc^{n} \longrightarrow \bc^{m}$ be a bicomplex holomorphic map.
    \begin{enumerate}
        \item $F$ is bicomplex holomorphic on each bicomplex variable separately.
        \item $F$ has an idempotent representation
                \begin{align*}
                    F(z_{1},z_{2}) = f_{1}(z_{1})\ebf + f_{2}(z_{2})\ebft,
                \end{align*}
              where $f_{i} : \mathcal{U}_{i} \subset \co^{n} \longrightarrow \co^{m}$ is complex holomorphic, for $i=1,2$.
        \item The derivative of $F$ has an idempotent representation
                \begin{align*}
                    DF_{(z_{1},z_{2})} = (Df_{1})_{z_{1}}\ebf + (Df_{2})_{z_{2}}\ebft.
                \end{align*}
    \end{enumerate}
\end{Cor}
\begin{proof}
    It is enough to prove the statement for $m=1$. The first item follows from Theorem \ref{form} and the last two assertions from item 1 and the analogous results for $n=1$ in \cite[Theorems 7.6.3 and Corollary 7.6.6]{Luna2015}.
\end{proof}

\begin{Cor}[Hartogs' Theorem for bicomplex maps]\label{bicor2}
    If $F : \mathcal{U} \subset \bc^{n} \longrightarrow \bc$ is bicomplex holomorphic on each bicomplex variable, then $F$ is bicomplex holomorphic.
\end{Cor}
    \begin{proof}
        Let us consider the idempotent representation of the limit in \eqref{diflimit}. The hypothesis implies that $f_{1}(z_{1})$ and $f_{2}(z_{2})$ are complex holomorphic functions on each $z_{ij}, z_{ij}$, respectively, where $z_{i} = (z_{i1}, \dots, z_{in}) \in \co^{n}$ and $i=1,2$. By Hartogs' Theorem for complex variables, $f_{1}$ and $f_{2}$ are complex holomorphic functions on $\mathcal{U}_{1}$ and $\mathcal{U}_{2}$, hence $F$ is bicomplex holomorphic on $\mathcal{U}$.
    \end{proof}
For $Z = \lambda_{1} + \ju \lambda_{2} \in \bc$, from equations \eqref{cjt}, \eqref{cjh}, and \eqref{cj}, we obtain the following complex differential operators:
\begin{equation}\label{operators1}
\begin{aligned}
\frac{\partial}{\partial {Z}} &= \frac{1}{2} \left(\frac{\partial}{\partial \lambda_1} +\ju \frac{\partial}{\partial \lambda_2}\right),\\
\frac{\partial}{\partial \cjh{Z}} &= \frac{1}{2} \left(\frac{\partial}{\partial \lambda_1} -\ju \frac{\partial}{\partial \lambda_2}\right),\\
\frac{\partial}{\partial \cjt{Z}} &= \frac{1}{2} \left(\frac{\partial}{\partial \cj{\lambda}_1} -\ju \frac{\partial}{\partial \cj{\lambda}_2}\right),\\
\frac{\partial}{\partial \cj{Z}} &= \frac{1}{2} \left(\frac{\partial}{\partial \cj{\lambda}_1} +\ju \frac{\partial}{\partial \cj{\lambda}_2}\right).
     \end{aligned}
\end{equation}
By writting $Z_{i} = x_{i} + \iu y_{i} + \ju v_{i} + \ku t_{i} \in \bc \simeq \re^{4}$, the above equations yield the following real differential operators:
\begin{equation}\label{operators2}
\begin{aligned}
    \frac{\partial }{\partial Z_{i}} &= \frac{1}{4}\left(\frac{\partial }{\partial x_{i}} + \iu \frac{\partial }{\partial y_{i}} + \ju\frac{\partial }{\partial v_{i}} +\ku \frac{\partial }{\partial t_{i}} \right),\\
    \frac{\partial }{\partial \cjt{Z}_{i}} &= \frac{1}{4}\left(\frac{\partial }{\partial x_{i}} + \iu \frac{\partial }{\partial y_{i}} - \ju\frac{\partial }{\partial v_{i}} -\ku \frac{\partial }{\partial t_{i}}\right), \\
    \frac{\partial }{\partial \cjh{Z}_{i}} &= \frac{1}{4}\left( \frac{\partial }{\partial x_{i}} - \iu \frac{\partial }{\partial y_{i}} - \ju\frac{\partial }{\partial v_{i}} -\ku \frac{\partial }{\partial t_{i}}\right), \\
    \frac{\partial }{\partial \cj{Z}_{i}} &= \frac{1}{4}\left(\frac{\partial }{\partial x_{i}} - \iu \frac{\partial }{\partial y_{i}} + \ju\frac{\partial }{\partial v_{i}} -\ku \frac{\partial }{\partial t_{i}}\right).
\end{aligned}
\end{equation}
Now, a bicomplex-valued map $F : \mathcal{U} \subset \bc^{n} \longrightarrow \bc$ can be seen as a real map
\begin{equation*}
 F = (F^{1}, F^{2}, F^{3}, F^{4}): \mathcal{U} \subset \re^{4n} \longrightarrow \re^{4},   
\end{equation*}
where $F^{i}$ is a real-valued function. It holds that:
\begin{align*}
    \frac{\partial F}{\partial Z_{i}} = \frac{\partial F^{1}}{ \partial Z_{i}} + \iu \frac{\partial F^{2}}{\partial Z_{i}} + \ju \frac{\partial F^{3}}{\partial Z_{i}} + \ku \frac{\partial F^{4}}{\partial Z_{i}},
\end{align*}
and similarly for the derivatives $\frac{\partial F}{\partial \cjt{Z}_{i}}, \frac{\partial F}{\partial \cjh{Z}_{i}}, \frac{\partial F}{\partial \overline{Z}_{i}}$.

\begin{Cor}\label{vanish}
    Let $F : \mathcal{U} \subset \bc^{n} \longrightarrow \bc$ be a continuous bicomplex valued function. Then $F$ is bicomplex holomorphic if and only if 
        \begin{align*}
            \frac{\partial F}{\partial \cjt{Z}_{i}}, \frac{\partial F}{\partial \cjh{Z}_{i}}, \frac{\partial F}{\partial \cj{Z}_{i}} \equiv 0,
        \end{align*}
    on $\mathcal{U}$ for all $i=1, \dots, k$.
\end{Cor}
\begin{proof}
    This is a consequence of item 1 of Corollary \ref{bicor1} and the analogous statements for $n=1$ in \cite[Theorems 7.4.3 and 7.6.5]{Luna2015}.
\end{proof}

The next results are immediate consequences of the idempotent representations of bicomplex holomorphic maps and their derivatives in Corollary \ref{bicor1} and the analogous statements for complex holomorphic maps.

\begin{The}[Chain rule]
    Let $F = (F^{1}, \dots, F^{m}) : \mathcal{U} \subset \bc^{n} \longrightarrow \bc^{m}$ and $G : \mathcal{V} \subset \bc^{m} \longrightarrow \bc^{k}$ be bicomplex holomorphic maps such that $F(\mathcal{U}) \cap \mathcal{V} \neq \emptyset$. Then $G \circ F$ is bicomplex differentiable at each $Z \in \mathcal{U}$ and 
    \begin{align*}
        D(G \circ F)_{Z} = DG_{F(Z)} \circ DF_{Z}.
    \end{align*}
\end{The}

\begin{Def}
    Let $F : \mathcal{U} \subset \bc^{n} \longrightarrow \bc^{m}$ be a bicomplex holomorphic map with $n \geq m$. We say that $Z \in \mathcal{U}$ is a bicomplex singular point if $\rk (DF_{Z}) \neq (m,m)$. Otherwise, $Z$ is called a bicomplex regular point.
\end{Def}

If $F(z_{1},z_{2}) = f_{1}(z_{1})\ebf + f_{2}(z_{2})\ebft$ is the idempotent representation of $F$ and $W = w_{1}\ebf + w_{2}\ebft$ is a bicomplex singular point, then $w_{1}$ or $w_{2}$ are singular points of $f_{1}$ or $f_{2}$, respectively. 

\begin{The}[Implicit map]
    Let $F : \mathcal{U} \subset \bc^{n} \longrightarrow \bc^{m}$ be a bicomplex holomorphic map such that $m < n$. Suppose that for some $A \in \mathcal{U}$ one has $F(A) = 0$ and $\rk(DF_{W}'') = (m,m)$, where $DF_{W}''$ is the $m \times m$ block of $DF_{W}$ relative to the variables $Z_{n-m+1}, \dots, Z_{n}$. Then there exists an open set $\mathcal{U}' = \mathcal{V} \times \mathcal{W} \subset \bc^{n-m} \times \bc^{m}$ and a bicomplex holomorphic map $G : \mathcal{V} \longrightarrow \mathcal{W}$ such that 
    \begin{align*}
        F(Z', G(Z')) = 0 \quad \text{for all} \quad Z' \in \mathcal{V} \quad \text{and} \quad G(A') = A.
    \end{align*}
\end{The}

\begin{The}[Inverse map]
    Let $F: \mathcal{U} \subset \bc^{n} \longrightarrow \bc^{n}$ be a bicomplex holomorphic map such that $A \in \mathcal{U}$ is a bicomplex regular point. Then there exist open sets $\mathcal{V} \subset \bc^{m}$, $\mathcal{W} \subset \mathcal{U}$, and a bicomplex holomorphic inverse map $G : \mathcal{V} \longrightarrow \mathcal{W}$ of $F$ such that 
    \begin{align*}
        DG_{F(A)} \circ DF_{Z} = DF_{G(A)} \circ DF_{A} = \idt_{\bc^{n}}.
    \end{align*}
\end{The}

\section{Bicomplex holomorphic Milnor fibration}\label{s3}

In this section, we prove a Milnor fibration type theorem for bicomplex holomorphic functions. The proof follows from slightly modifying Milnor's proof in \cite{Milnor1968}, which is also detailed in \cite{Cisneros2018}. Let us fix some notations. A bicomplex holomorphic function $F :(\bc^{n},0) \longrightarrow (\bc,0)$ has an idempotent representation 
\begin{align*}
    F(z_{1},z_{2}) = f_{1}(z_{1})\ebf + f_{2}(z_{2})\ebft,
\end{align*}
where $f_{1}$ and $f_{2}$ are holomorphic functions on separable variables, according to Corollary \ref{bicor1}. Moreover, in a small neighbourhood $\mathcal{V}$ of $\bc^{*}$, $f_{1}$ and $f_{2}$ have isolated critical values in $0$, respectively, by Bertini-Sard theorem. We consider a different norm on $\bc^{n}$ as follows. For $Z = z_{1}\ebf + z_{2}\ebft \in \bc^{n}$ we define
\begin{align*}
    \nm Z \nm_{\bc}^{2} = \nm z_{1}\nm^{2} + \nm z_{2} \nm^{2}.
\end{align*}
The bicomplex sphere will be
\begin{align*}
    \mathbb{S}_{\bc, \epsilon}^{4n-1} = \{ Z \in \bc^{n} : \nm Z \nm_{\bc} = \epsilon\}.
\end{align*}
The coordinate change from the idempotent to the polar representation will be denoted by 
\begin{align*}
    \Phi: \bc^{*} \longrightarrow \bc^{*}.
\end{align*}
The trigonometric representation of $F(Z)$ is $F(Z) = \nm F(Z) \nm_{\iu}e^{\ju\Theta(Z)}$ and the intersection $K_{\epsilon} := F^{-1}(\zd) \cap  \mathbb{S}_{\bc,\epsilon}^{4n-1}$ is called the \textit{bicomplex link}.

\begin{The}\label{milnfibr}
    Let $F: (\bc^{n},0) \longrightarrow (\bc,0)$ be a bicomplex holomorphic function germ. Then the map
    \begin{align}\label{milnfib}
        \varphi := \frac{F(Z)}{\nm F(Z)\nm_{\iu}} :  \mathbb{S}_{\bc,\epsilon}^{4n-1} \setminus K_{\epsilon} \longrightarrow \cc
    \end{align}
    is a locally trivial smooth fibration for every sufficiently small $\epsilon>0$.
\end{The}

\begin{Rem}
    \normalfont
    One may consider the hyperbolic norm $\nm \cdot \nm_{\ku}$ instead of the complex one. This reduces to a product of the Milnor fibrations of $f_{1}$ and $f_{2}$ on the base space $\tr$. We shall see that, from the complex norm, we obtain new and interesting constructions. 
\end{Rem}

Theorem \ref{milnfibr} will be a consequence of the following three lemmas. 

\begin{Lem}\label{l1}
    The map $\varphi$ is a submersion.
\end{Lem}
    \begin{proof}
       The map $\varphi$ can be seen as the composition $\pi_{\iu} \circ \Phi \circ F(u,v)$, where $F$ is taken in the idempotent representation. Thus, it is enough to prove that 
        \begin{align*}
            F : \mathbb{S}_{\bc,\epsilon}^{4n-1}\setminus K_{\epsilon}& \longrightarrow \bc^{*}\\
                (z_{1},z_{2}) &\longmapsto f_{1}(z_{1})\ebf + f_{2}(z_{2})\ebft
        \end{align*}
        is a submersion. By \cite[Corollary 3.11]{Cisneros2018} the fibres of $f_{1}$ and $f_{2}$ are transversal to $\mathbb{S}_{\epsilon/\sqrt{2}}^{2n-1}$ for every sufficiently small $\epsilon$. More precisely, for $p = p_{1}\ebf + p_{2}\ebft \in \mathcal{V}$, one has
        \begin{align*}
            T_{p_{1}}f_{1}^{-1}(p_{1}) + T_{p_{1}}\mathbb{S}_{\epsilon/\sqrt{2}}^{2n-1} &= \co^{n}\ebf, \\
            T_{p_{2}}f_{2}^{-1}(p_{2}) + T_{p_{2}}\mathbb{S}_{\epsilon/\sqrt{2}}^{2n-1} &= \co^{n}\ebft.
        \end{align*}
        Since $$\left(T_{p_{1}}\mathbb{S}_{\epsilon/\sqrt{2}}^{2n-1}\right)\ebf + \left(T_{p_{2}}\mathbb{S}_{\epsilon/\sqrt{2}}^{2n-1}\right)\ebft\subset T_{p}\left(\mathbb{S}^{4n-1}_{\bc,\epsilon}\right),$$ the statement follows.       
    \end{proof}
From now on, we consider the vector field $\mathbb{V}$ in $\cc$ given by $\mathbb{V}(e^{\ju\Theta}) = \ju e^{\ju\Theta}$.
\begin{Lem}\label{l3}
    There exists a complete vector field $\mathbb{W}$ on $\mathbb{S}_{\epsilon}^{4n-1}\setminus K_{\epsilon}$ that projects to $\mathbb{V}$ by $\varphi$. 
\end{Lem}
    \begin{proof}
        Milnor proved the existence of a vector field $\mathbb{W}_{i}$ on $\mathbb{S}_{\epsilon/\sqrt{2}}^{2n-1}$ associated with the spherical fibration of $\Theta_{i} := f_{i}/\nm f_{i} \nm$, where $i = 1,2$ (see \cite[Lemma 3.14]{Cisneros2018}).  We shall see that $\mathbb{W}(z_{1},z_{2}) = \mathbb{W}_{1}(z_{1})\ebf + \mathbb{W}_{2}(z_{2})\ebft$ is the desired vector field. As before, since $\mathbb{W}_{i}$ is complete and tangent to $\mathbb{S}_{\epsilon/\sqrt{2}}^{2n-1}$, it follows that $\mathbb{W}$ is complete and tangent to $\mathbb{S}_{\epsilon/\sqrt{2}}^{2n-1}\ebf + \mathbb{S}_{\epsilon/\sqrt{2}}^{2n-1}\ebft$. Moreover, the integral curve $p_{i}(t)$ of $\mathbb{W}_{i}$ projects to the path $t \in \mathbb{S}^{1}$. Note that $p(t) = p_{1}(t)\ebf + p_{2}(t)\ebft$ is the integral curve of $\mathbb{W}$. Lastly, one can see that
        \begin{align*}
            \Theta\left(p(t)\right) &= \Phi \circ \left(\Theta_{1}(p_{1}(t)), \Theta_{2}(p_{2}(t)\right) = t
        \end{align*}
        for every small $t> 0$.
    \end{proof}
\begin{Lem}\label{l2}
    The map $\varphi$ is locally trivial.
\end{Lem}
    \begin{proof}
         On a sufficiently small neighbourhood in $\cc$, we identify $e^{\ju\Theta} \simeq \Theta \in \co$. By Lemma \ref{l3},  there exists a lifting $\mathbb{W}$ in $\mathbb{S}_{\bc,\epsilon}^{4n-1}\setminus K_{\epsilon}$ of $\mathbb{V}$ by the map $\varphi$ that is also complete. Let us denote by $H_{t}(Z)$ the integral curve of $\mathbb{W}$ passing through $Z$. The following map defines a local trivialization for $\varphi$:
        \begin{align*}
            \mathcal{U} \times \varphi^{-1}(\Theta_{0}) & \longrightarrow \varphi^{-1}(\mathcal{U}) \\
            \left(\Re(\Theta+\Theta_{0}), \Im(\Theta+\Theta_{0}), Z\right) &\longmapsto H_{\Re(\Theta)} \circ H_{\Im(\Theta)} \circ Z,
        \end{align*}
        where $\mathcal{U} = \{ \Theta \in \cc : \nm \Theta + \Theta_{0}\nm < \delta\}$ is a small neighbourhood of $\Theta_{0}$. 
    \end{proof}

Observe that, in general, we cannot decompose $\Theta(Z)$ as $\Phi \circ (\Theta_{1}\ebf+\Theta_{2}\ebft)$. We discuss now some topological properties of $\varphi$. First, notice that,
\begin{align*}
    K_{\epsilon} = \left\{ Z \in \bc^n : \nm Z \nm_{\bc} = \epsilon, \; F(Z) \notin \bc^{*}\right\}.
\end{align*}
Recall that the map $F$ has the following idempotent representation 
\begin{align*}
    F(z_{1},z_{2}) = f_{1}(z_{1})\ebf + f_{2}(z_{2})\ebft,
\end{align*}
Thus, $K_{\epsilon}$ can be identified with the link of the holomorphic function $f(z_{1},z_{2}): \co^{2n} \longrightarrow \co$ given by $f(z_{1},z_{2}) = f_{1}(z_{1})f_{2}(z_{2})$. This gives the following.
\begin{Cor}[\cite{Milnor1968}, Theorem 5.2]
    The bicomplex link $K_{\epsilon}$ is the link of the holomorphic function $f(z_{1},z_{2})$ and thus it is $(2n-2)$-connected.
\end{Cor}
In addition, a typical fiber is described by
\begin{align*}
    \varphi^{-1}(1) = \left\{ Z \in \bc^n : \nm Z \nm_{\bc} = \epsilon, \; F(Z) \in \bc^{*}, F(Z) = \nm Z \nm_{\iu}\right\}.
\end{align*}
Consider the following holomorphic functions on $\mathcal{U}$:
\begin{align*}
    h_{1}(z_{1},z_{2}) &= f_{1}(z_{1}) + f_{2}(z_{2}), \\
    h_{2}(z_{1},z_{2}) &= f_{1}(z_{1}) - f_{2}(z_{2}).
\end{align*}
Then $h_{2}(z_{1},z_{2}) \equiv 0$ and $h_{1}(z_{1},z_{2}) \neq 0$ for all points in $\varphi^{-1}(1)$. Equivalently, $\varphi^{-1}(1) = K_{2,\epsilon} \setminus K_{1,\epsilon}$, where $K_{i,\epsilon}$ is the usual link of the holomorphic function $h_{i}(z_{1},z_{2})$, for $i=1,2$. 

\begin{Cor}
    Assume further that the holomorphic functions $h_{1}$ and $h_{2}$ have isolated singularity at the origin and have no common irreducible components. Then the fundamental group of $\varphi^{-1}(1)$ is isomorphic with $\mathbb{Z}$ and the other homotopy groups are the same as those of a bouquet of spheres of dimension $2n-2$.   
\end{Cor}
    \begin{proof}
       By the works of Lê, see \cite{Cisneros2009} for instance, the maps
       \begin{align*}
           \frac{h_{1}}{\nm h_{1}\nm} :& K_{2,\epsilon}\setminus K_{1,\epsilon} \longrightarrow \ci \\
           h_{1} :& h_{2}^{-1}(0) \cap h_{1}^{-1}(\partial \mathbb{D}_{\delta}) \cap \mathring{\mathbb{B}}_{\epsilon}^{4n} \longrightarrow \partial \mathbb{D}_{\delta}
       \end{align*}
       are isomorphic locally trivial fibrations for $\epsilon,\delta>0$ sufficiently small, where $ \mathbb{D}_{\delta}$ is the closed disk of radius $\delta > 0$ in $\co$ and $\mathring{\mathbb{B}}_{\epsilon}^{4n}$ is the open ball of radius $\epsilon > 0$ in $\bc^{n}$, both centered at origin. Furthermore, it is well-known that the homotopy type of these fibers is that of a bouquet of spheres of dimension $2n-2$, by \cite[\S 5.7 and 5.8]{Looijenga2013}. The conclusion follows by applying the exact sequence relating the homotopy groups of the total, base, and fiber spaces of a fibration.
    \end{proof}

\begin{The}[Tube fibrations] \hfill 
    \begin{enumerate}
        \item There exists $\epsilon > 0$ and $\delta = \delta(\epsilon) > 0$ such that the restriction
             \begin{align}\label{milnfibtube1}
                 F : \mathcal{N}(\epsilon,\delta) \longrightarrow \bc^{*} \cap \mathbb{B}^{4}_{\delta}
             \end{align}
             is a locally trivial fibration, where $\mathcal{N}(\epsilon,\delta) = \mathbb{B}_{\epsilon}^{4n} \cap F^{-1}(\bc^{*}\cap \mathbb{B}^{4}_{\delta})$ and $\mathbb{B}_{\epsilon}^{4n}, \mathbb{B}^{4}_{\delta}$ are closed balls centered at the origin in $\bc^{n}, \bc$ with radius $\epsilon, \delta>0$, respectively. 
        \item There exists $\epsilon > 0$ and $\delta = \delta(\epsilon) > 0$ such that the restriction
             \begin{equation}\label{milnfibtube2}
                 F : \mathcal{N}_{\iu}(\epsilon,\delta) \longrightarrow \cc
             \end{equation}
             is a locally trivial fibration, where $\mathcal{N}_{\iu}(\epsilon,\delta) = \mathbb{B}_{\epsilon}^{4n} \cap F^{-1}(\cc)$ and $\mathbb{B}_{\epsilon}^{4n}$ is the closed ball centered at the origin in $\bc^{n}$ with radius $\epsilon > 0$.
    \end{enumerate}
\end{The}
        \begin{proof}
            The first map is a submersion by the $a_{f}$-property of the holomorphic functions $f_{1}(z_{1})$ and $f_{2}(z_{2})$ as we argued in Lemma \ref{l1}. The same is applied to the second map, where we consider the decomposition $\pi_{\iu} \circ \Phi \circ F$. The conclusion follows from the relative Ehresmann fibration theorem.
        \end{proof}

\begin{Rem}\label{rem}
    Notice that the total spaces of \eqref{milnfib} and \eqref{milnfibtube2} have distinct dimensions. Therefore, the tube and the spherical bicomplex fibrations cannot be equivalent, which contrasts with the complex case. 
\end{Rem}

\section{Mixed polynomials}\label{s4}
In this section, we introduce the notion of mixed polynomials in bicomplex variables and the notion of polar weighted homogeneity. We present basic properties and examples motivated by the case of mixed polynomials on complex variables studied in \cite{Cisneros2008} and \cite{Oka2008}.

A \textit{bicomplex mixed monomial in a variable} $Z_{i} \in \bc$ is bicomplex valued-function $\bc \longrightarrow \bc$ of the form
\begin{align*}
    Z_{i}(a_{i}, b_{i}, c_{i},d_{i}) := Z_{i}^{a_{i}}\cjt{Z}_{i}^{b_{i}}\cjh{Z}_{i}^{c_{i}}\cj{Z}_{i}^{d_{i}},
\end{align*}
where $a_{i}, b_{i}, c_{i},d_{i}$ are non-negative integers. Moreover, a \textit{bicomplex mixed monomial} is a bicomplex valued-function $\bc^{n} \longrightarrow \bc$ of the form
\begin{align*}
    Z(\nu, \mu_{1}, \mu_{2},\mu_{3}) := \prod_{i=1}^{n}Z_{i}(\nu^{i}, \mu_{1}^{i}, \mu_{2}^{i}, \mu_{3}^{i}),
\end{align*}
where $\nu = (\nu^{i})$, $\mu_{j} = (\mu_{j}^{i})$ are vectors of non-negative integers, for $j = 1,2,3$. Lastly, a \textit{bicomplex mixed polynomial} is a finite sum of the form
\begin{align*}
    F(Z) = \sum_{\nu,\mu}\lambda_{\nu,\mu}Z(\nu,\mu_{1},\mu_{2},\mu_{3}),
\end{align*}
where $\lambda_{\nu,\mu} \in \bc\setminus\{0\}$.

For simplicity, we shall denote $F(Z)$ by $F(Z,\cjt{Z}), F(Z,\cjh{Z}),$ and $F(Z,\cj{Z})$ if $F$ depends only on the respective variables and we refer to it by tilde, hat, or bar-mixed polynomials, respectively. A \textit{singular point} of a bicomplex mixed polynomial is a singular point of the associated real polynomial map $\re^{4n} \longrightarrow \re^{4}$.

    \subsection{Idempotent representations}
    
    The idempotent representation of bicomplex numbers allows us to reduce mixed polynomials to complex maps whose coordinates are certain polynomials of the same type as $F(Z)$.
    
    \begin{Prop}\label{p21}
        Let $F(Z) = \sum_{\nu,\mu}\lambda_{\nu,\mu}Z(\nu,\mu_{1},\mu_{2},\mu_{3})$ be a bicomplex mixed polynomial, where $\lambda_{\nu,\mu} = \lambda_{\nu,\mu}^{1}\ebf + \lambda^{2}_{\nu,\mu}\ebft$ for each multi-index $\nu,\mu$. Then, up to a linear change of coordinates, $F : \co^{2n} \longrightarrow \co^{2}$ has the following form:
        \begin{align*}
            F(z_{1},\bar{z}_{1}, z_{2}, \bar{z}_{2}) = \sum_{\nu,\mu}\lambda_{\nu,\mu}^{1}z_{1}^{\nu}\bar{z}_{1}^{\mu_{1}}z_{2}^{\mu_{2}}\bar{z}_{2}^{\mu_{3}}\ebf + \sum_{\nu,\mu}\lambda^{2}_{\nu,\mu}z_{2}^{\nu}\bar{z}_{2}^{\mu_{1}}z_{1}^{\mu_{2}}\bar{z}_{1}^{\mu_{3}}\ebft.
        \end{align*}
        We shall denote $F(z_{1},z_{2}) = (f_{1}(z_{1},z_{2}), f_{2}(z_{1},z_{2}))$ and call it the idempotent representation of $F(Z)$.
    \end{Prop}
        \begin{proof}
            Let $Z = (Z_{1}, \dots, Z_{n}) \in \bc^{n}$ and, for each $k$, write $Z_{k} =\lambda_{1k} + \ju \lambda_{2k}$, where $\lambda_{lk} \in \co$ for $l =1,2$. Every summand in the decomposition of $F$ has the form
            \begin{align}\label{eq1}
                \lambda_{\nu,\mu}Z(\nu,\mu_{1},\mu_{2},\mu_{3}) = \lambda_{\nu,\mu}Z_{1}^{\nu^{1}}\dots Z_{n}^{\nu^{n}}\cjt{Z}_{1}^{\mu_{1}^{1}}\dots\cjt{Z}_{n}^{\mu_{1}^{n}}\cjh{Z}_{1}^{\mu_{2}^{1}}\dots \cjh{Z}_{n}^{\mu_{2}^{n}}\cj{Z}_{1}^{\mu_{3}^{1}}\dots\cj{Z}_{n}^{\mu_{3}^{n}}.
            \end{align}
        Now, we consider the idempotent representation of $Z_{k}$ and define the following linear coordinate change $\psi : \co^{2n} \longrightarrow \co^{2n}$ by 
        $$\psi(z_{11}, z_{12}, \dots, z_{1n}, z_{2n}) = (u_{11}, v_{21}, \dots, u_{1n}, v_{2n}),$$ where $z_{1k} = \lambda_{1k} - \iu \lambda_{2k}, \; z_{2k} = \lambda_{1k} + \iu \lambda_{2k}$
        for all $k = 1, \dots, n$. The product \eqref{eq1} is taken with respect to the basis $(\ebf,\ebft)$, and by writing $F$ on these coordinates we obtain the result. 
        \end{proof}
        The hat, tilde, and bar-mixed cases are those for which all the vectors $\mu_{i},\mu_{j}$ vanish for some pair $i,j = 1,2,3$ and we thus obtain:       
    \begin{Cor}\label{c21}
        Up to a linear change of coordinates, the following statements hold true:
        \begin{enumerate}
            \item A tilde-mixed polynomial $F(Z,\cjt{Z})$ has an idempotent representation
                \begin{align*}
                    F(z_{1},z_{2},\bar{z}_{1}, \bar{z}_{2}) = \sum_{\nu,\mu_{1}}\lambda_{\nu,\mu_{1}}^{1}z_{1}^{\nu}\bar{z}_{1}^{\mu_{1}}\ebf + \sum_{\nu,\mu_{1}}\lambda_{\nu,\mu_{1}}^{2}z_{2}^{\nu}\bar{z}_{2}^{\mu_{1}}\ebft.
                \end{align*}
            \item A hat-mixed polynomial $F(Z,\cjh{Z})$ has an idempotent representation
                \begin{align*}
                    F(z_{1},z_{2},\bar{z}_{1}, \bar{z}_{2}) = \sum_{\nu,\mu_{2}}\lambda_{\nu,\mu_{2}}^{1}z_{1}^{\nu}z_{2}^{\mu_{2}}\ebf + \sum_{\nu,\mu_{2}}\lambda_{\nu,\mu_{2}}^{2}z_{2}^{\nu}z_{1}^{\mu_{2}}\ebft.
                \end{align*}
            \item A bar-mixed polynomial $F(Z,\cj{Z})$ has an idempotent representation
                \begin{align*}
                    F(z_{1},z_{2},\bar{z}_{1}, \bar{z}_{2}) = \sum_{\nu,\mu_{3}}\lambda_{\nu,\mu_{3}}^{1}z_{1}^{\nu}\bar{z}_{2}^{\mu_{3}}\ebf +  \sum_{\nu,\mu_{3}}\lambda_{\nu,\mu_{3}}^{2}z_{2}^{\nu}\bar{z}_{1}^{\mu_{3}}\ebft.
                \end{align*}
        \end{enumerate}
    \end{Cor}

    \subsection{Polar weighted homogeneous property}\label{s4.2}

    \begin{Def}\label{acdef3}
    Let $p_{j},u_{j}, t_{j}$ with $j = 1, \dots, n$ be positive integers such that
    \begin{align*}
        \gcd(p_{1}, \dots, p_{n}) = \gcd(u_{1}, \dots, u_{n}) = \gcd(t_{1}, \dots, t_{n}) = 1.
    \end{align*}
    For each $\Lambda \in \bc^{*}$ consider the polar form $\Lambda = s e^{\iu\theta}e^{\ju\Theta}$, where $s \in \re^{+}, e^{\iu\theta} \in \ci$ and $e^{\ju\Theta} \in \cc$. A polar $\bc^{*}$-action with radial weights $(t_{1}, \dots, t_{n})$, polar weights $(p_{1}, \dots, p_{n})$, and complex polar weights $(u_{1}, \dots, u_{n})$ is given by
        \begin{align*}
        \Lambda \cdot Z &= \left( s^{t_{1}}e^{\iu p_{1}\theta}e^{\ju u_{1}\Theta}Z_{1}, \dots, s^{t_{n}}e^{\iu p_{n}\theta}e^{\ju u_{n}\Theta}Z_{n}\right), \\
        \Lambda \cdot \cjh{Z} &= \left( s^{t_{1}}e^{\iu p_{1}\theta}e^{-\ju u_{1}\Theta}\cjh{Z}_{1}, \dots,  s^{t_{n}}e^{\iu p_{n}\theta}e^{-\ju u_{n}\Theta}\cjh{Z}_{n}\right), \\
        \Lambda \cdot \cjt{Z} &= \left(s^{t_{1}}e^{-\iu p_{1}\theta}e^{-\ju u_{1}\overline{\Theta}}\cjt{Z}_{1}, \dots,  s^{t_{n}}e^{-\iu p_{n}\theta}e^{-\ju u_{n}\overline{\Theta}}\cjt{Z}_{n}\right), \\
        \Lambda \cdot \cj{Z} &= \left(s^{t_{1}}e^{-\iu p_{1}\theta}e^{\ju u_{1}\overline{\Theta}}\cj{Z}_{1}, \dots, s^{t_{n}}e^{-\iu p_{n}\theta}e^{\ju u_{n}\overline{\Theta}}\cj{Z}_{n}\right).
    \end{align*}
\end{Def}

Notice that this action is a combination of $\re^{+}, \mathbb{S}^{1}$, and $\mathbb{S}^{1}_{\co}$-actions with weights.

    \begin{Def}\label{poldef3}
        Let $F: \bc^{n} \longrightarrow \bc$ be a bicomplex mixed polynomial, $a,c,d,d'$ non-negative integers such that $a,c > 0$ and $d > 0$ or $d' > 0$. We say that $F$ is polar weighted homogeneous with radial weight type $(t_{1},\dots, t_{n};a)$, polar weight type $(p_{1}, \dots, p_{n};c)$, and complex polar weight type $(u_{1}, \dots, u_{n};d,d')$ if the following identity holds:
        \begin{align}\label{defipol}
            F\left(se^{\iu\theta}e^{\ju\Theta} \cdot (Z, \cjt{Z}, \cjh{Z}, \cj{Z})\right) = s^{a}e^{\iu c\theta}e^{\ju d\Theta}e^{\ju d'\overline{\Theta}}F(Z, \cjt{Z}, \cjh{Z}, \cj{Z}), \;\; s \in \re^{+}, \; e^{\iu\theta} \in \ci, \; e^{\ju\Theta}, e^{\ju\overline{\Theta}} \in \cc,
        \end{align}
        where $se^{\iu\theta}e^{\ju\Theta} \cdot (Z, \cjt{Z}, \cjh{Z}, \cj{Z})$ denotes the previous polar action of $\bc^{*}$.
    \end{Def}

    \begin{Rem}
        \normalfont
        Notice that the complex mixed polynomials in the idempotent representation $F(z_{1},z_{2}) = (f_{1}, f_{2})$ are polar weighted homogeneous in the sense of \cite{Cisneros2008} of the same radial and polar types as $F$ but not simultaneously on the same variables in the case of the polar actions.
    \end{Rem}

    Before we proceed with examples, we state a property induced by the radial action on the discriminant.

\begin{Prop}\label{p22}
    Let $F(Z)$ be a bicomplex polar weighted homogeneous polynomial. If $P \in \bc^{n}$ is a critical point of $F$, then $s \cdot P$ is also a critical point for all $s \in \re^{+}$. In particular, the discriminant $\Delta_{F}$ of $F$ consists of a union of lines passing through the origin. 
\end{Prop}
    \begin{proof}
        By linearity, we may prove the statement for the idempotent representation $F = (f_{1}(z_{1},z_{2}), f_{2}(z_{1},z_{2}))$. One has that $f_{1}$ and $f_{2}$ are radial weighted homogeneous of the same type as $F$. The Jacobian matrix $J_{F}$ of $F$ can be written in terms of mixed derivatives in relation to $z_{1},\cj{z}_{1}, z_{2}, \cj{z}_{2}$. This implies that each $j$-column $J_{F}^{j}$ of this matrix satisfies
        \begin{align*}
            \left( J_{F}^{j}\right)^{T}(s\cdot P) = s^{a-t_{j}}J_{F}^{j}(P),
        \end{align*}
        where $t_{j}$ is the radial weight associated to $z_{1j},z_{2j}$ (see \cite[Section 3.1]{Cisneros2008}). Hence, the $\re^{+}$-action preserves the rank of $F$. To conclude, suppose that $F$ is radial weighted homogeneous with weight $d$. Let $X \in \Delta_{F}$ and $s \in \re^{+}$, then there exists a singular point $P$ such that $F(P) = X$ so that $F(s^{1/d}\cdot P) = sX$ and the statement follows.
    \end{proof}
    
See also \cite[Lemma 2.16]{Cisneros2017} for a similar statement applied to mixed polynomials. The next identities follow by taking the derivate of\eqref{defipol} with respect to $s$, $\theta$, $\Theta$, and $\overline{\Theta}$.

\begin{align*}
    aF(Z) &= \sum_{i=1}^{n}t_{i}\left[ Z_{i}\frac{\partial F}{\partial Z_{i}} + \cjt{Z}_{i}\frac{\partial F}{\partial \cjt{Z}_{i}} + \cjh{Z}_{i}\frac{\partial F}{\partial \cjh{Z}_{i}} + \cj{Z}_{i}\frac{\partial F}{\partial \cj{Z}_{i}}\right], \\
    cF(Z) &= \sum_{i=1}^{n}p_{i}\left[Z_{i}\frac{\partial F}{\partial Z_{i}} + \cjt{Z}_{i}\frac{\partial F}{\partial \cjt{Z}_{i}} - \cjh{Z}_{i}\frac{\partial F}{\partial \cjh{Z}_{i}} - \cj{Z}_{i}\frac{\partial F}{\partial \cj{Z}_{i}}\right], \\
    dF(Z) &= \sum_{i=1}^{n}u_{i}\left[Z_{i}\frac{\partial F}{\partial Z_{i}} - \cjh{Z}_{i}\frac{\partial F}{\partial \cjh{Z}_{i}}\right], \\
    d'F(Z) &= \sum_{i=1}^{n}u_{i}\left[\cj{Z}_{i}\frac{\partial F}{\partial \cj{Z}_{i}} - \cjt{Z}_{i}\frac{\partial F}{\partial \cjt{Z}_{i}}\right].
\end{align*}

    \subsection{Examples}

        \begin{Exam}[Weighted homogeneous polynomials]
            \normalfont If $F(Z) = \sum_{\nu}\lambda_{\nu}Z^{\nu}$ is a bicomplex weighted homogeneous polynomial, then there exist integers $p_{1}, \dots, p_{n},d$ for which
            \begin{align*}
                F\left( \Lambda^{p_{1}}Z_{1},\dots, \Lambda^{p_{n}}Z_{n}\right) = \Lambda^{d}F(Z),
            \end{align*}
            for all $\Lambda \in \bc\setminus\{0\}$, in particular, for those in $\bc^{*}$. In this case, the weight $d' = 0$.
        \end{Exam}

        \begin{Exam}[Mixed Pham-Brieskorn polynomials]
            \normalfont
            Let $F(Z) = \sum_{i=1}^{n}Z_{i}(a_{i},b_{i},c_{i},d_{i})$, where $a_{i}, b_{i}, c_{i}, d_{i}$ are non-negative integers for all $i$. Notice that if $b_{i} = c_{i} = d_{i} = 0$ it resembles the Pham-Brieskorn polynomials on complex variables. Consider weights $(t_{1}, \dots, t_{n})$, $(p_{1}, \dots, p_{n})$, and $(q_{1}, \dots, q_{n})$. Let us study each monomial separately:
            \begin{align*}
                se^{\iu\theta}e^{\ju\Theta}\cdot Z_{i}(a_{i},b_{i},c_{i},d_{i}) = s^{t_{i}(a_{i}+b_{i}+c_{i}+d_{i})}e^{\iu p_{i}(a_{i}-b_{i}+c_{i}-d_{i})\theta} e^{\ju q_{i}(a_{i}-c_{i})\Theta}e^{\ju q_{i}(d_{i}-b_{i})\overline{\Theta}}Z_i.
            \end{align*}
            For each $i=1, \dots, n$, let us suppose that
            \begin{align}\label{cnd}
                a_{i} - b_{i} > d_{i} - c_{i}, \quad a_{i} > c_{i},\quad d_{i}-b_{i} \ge 0.
            \end{align}
            If $d_{i} - b_{i} > 0$, we require further that $a_{i}-c_{i} = d_{i} - b_{i}$. It follows that $F(Z)$ is polar weighted of radial, polar, and complex polar types
            \begin{align*}
                & \left(\frac{1}{a_{1}+b_{1}+c_{1}+d_{1}}, \dots, \frac{1}{a_{n}+b_{n}+c_{n}+d_{n}}; 1\right), \\
                & \left(\frac{1}{a_{1}-b_{1}+c_{1}-d_{1}}, \dots, \frac{1}{a_{n}-b_{n}+c_{n}-d_{n}}; 1\right), \\
                & \left(\frac{1}{a_{1}-c_{1}}, \dots, \frac{1}{a_{n}-c_{n}}; 1;1\right),
            \end{align*} 
            respectively. Since $\cjt{X}\cj{X} = 1$ for all $X \in \cc$, if $b_{i} = d_{i}$, $F(Z)$ reduces to a hat-mixed polar weighted polynomial of (equal) radial and polar, and complex polar types
            \begin{equation}\label{types}
                \begin{split}
                & \left(\frac{1}{a_{1}+c_{1}}, \dots, \frac{1}{a_{n}+c_{n}}; 1\right), \\
                & \left(\frac{1}{a_{1}-c_{1}}, \dots, \frac{1}{a_{n}-c_{n}}; 1;0\right).
                \end{split}
            \end{equation}
            
            If $d_{i} = 0$ and $b_{i} \neq 0$ for some $i$, we obtain a negative exponent in the complex polar action, and $F(Z)$ does not satisfy the polar property. Moreover, if $b_{i} = c_{i} = 0$, then $F(Z)$ is a bar-mixed polynomial which is polar weighted only if $d_{i} = 0$. Nevertheless, let us consider $d_{i} = c_{i} = 0$ for all $i$, so $F(Z)$ is a tilde-mixed polynomial. One has that $F$ admits an idempotent representation $F(z_{1},z_{2}) = (f_{1}(z_{1},\bar{z}_{1}), f_{2}(z_{2},\cj{z}_{2}))$, where $f_{1}$ and $f_{2}$ are polar weighted with the same radial and polar types. The action on each component may be expressed in bicomplex terms as follows. Let $r\lambda \in \bc^{*}$, where $r \in \mathbb{D}^{+}$ and $\lambda \in \tr$. For each monomial, one has 
            \begin{align*}
                r\lambda \cdot Z_{i}^{a_{i}}\cjt{Z}_{i}^{b_{i}} = r^{p_{i}(a_{i}+b_{i})}\lambda^{q_{i}a_{i}}\cjt{\lambda}^{q_{i}b_{i}}Z_{i}^{a_{i}}\cjt{Z}_{i}^{b_{i}} = r^{p_{i}(a_{i}+b_{i})}\lambda^{q_{i}(a_{i}-b_{i})}Z_{i}^{a_{i}}\cjt{Z}_{i}^{b_{i}},
            \end{align*}
            and according to this hyperbolic action, $F(Z, \cjt{Z})$ is polar weighted of the same types \eqref{types}, where we replace $c_{i}$ by $b_{i}$.
        \end{Exam}

        \begin{Exam}[Mixed cyclic polynomials]\label{cyclic}
            \normalfont 
            Let $F(Z) = \sum_{i=1}^{n-1}Z_{i}(a_{i}, b_{i}, c_{i}, d_{i})Z_{i+1} + Z_{n}(a_{n},b_{n},c_{n},d_{n})Z_{1}$ and assume the conditions \eqref{cnd} on $a_{i}, b_{i}, c_{i}, d_{i}$ for all $i = 1, \dots, n$. The square matrices of order $n$ associated with the radial, polar, and complex polar actions are upper triangular with rows
            \begin{align*}
                &\left(0, \dots, 0, a_{i}+b_{i}+c_{i}+d_{i}, 1, 0, \dots, 0\right), \\
                &\left(0, \dots, 0, a_{i}-b_{i}+c_{i}-d_{i}, 1, 0, \dots, 0\right), \\
                &\left(0, \dots, 0, a_{i}-c_{i}, 1, 0, \dots, 0\right),
            \end{align*}
            where $i=1, \dots, n$. These are all invertible so the associated systems always have non-trivial solutions and this implies $F(Z)$ is polar weighted homogeneous. The complex version of $F(Z)$ appears in the classification performed in \cite{Cisneros2017} of polar weighted polynomials on three variables with an isolated singularity at the origin.    
        \end{Exam}

        \begin{Exam}[Join]
            \normalfont
           Let $F(Z) = F(Z_{1}, \dots, Z_{n})$ and $G(W) = G(W_{1}, \dots, W_{m})$ bicomplex polar weighted polynomials with radial, polar, and complex polar weights $(t_{1}, \dots, t_{n}; a)$, $(t_{1}', \dots, t_{m}'; a')$, $(p_{1}, \dots, p_{n}; c)$, $(p_{1}', \dots, p_{m}'; c')$, $(u_{1}, \dots u_{n}; d,d')$, $(u_{1}', \dots, u_{m}'; e,e')$, respectively. Set
           \begin{align*}
               r = \gcd(a,a'), \;\; a_{1} &= \frac{a}{r}, \;\; a_{2} = \frac{a'}{r}, \\
               s = \gcd(c,c'), \;\; c_{1} &= \frac{c}{s}, \;\; c_{2} = \frac{c'}{s}, \\
               x = \gcd(d,e), \;\; d_{1} &= \frac{d}{x}, \;\; d_{2} = \frac{e}{x}, \\
               x' = \gcd(d', e') \;\; e_{1}', &= \frac{d'}{x}, \;\; e_{2}' = \frac{e'}{x}.
           \end{align*}
           Then $H(Z,W) = F(Z)+G(W)$ is polar weighted homogeneous of real, polar, and complex polar types
           \begin{align*}
               &\left(t_{1}a_{2}, \dots, t_{n}a_{2}, t_{1}'a_{1}, \dots, t_{m}'a_{1}; \text{lcm}(a_{1},a_{2}) \right), \\
               &\left(p_{1}c_{2}, \dots, p_{n}c_{2}, p_{1}'c_{1}, \dots, p_{m}'c_{1}; \text{lcm}(c_{1},c_{2}) \right), \\
               &\left(u_{1}d_{2}, \dots, u_{n}d_{2}, u_{1}'d_{1}, \dots, u_{m}'d_{1}; \text{lcm}(d_{1}, d_{2}), \text{lcm}(e_{1}', e_{2}')\right),
           \end{align*}
           respectively.
        \end{Exam}
        
        Motivated by \cite[Theorem 4.1]{Ruas2002} and \cite[Theorem 10.1]{Oka2008}, we have the following proposition relating mixed and holomorphic maps.
        
        \begin{Prop}\label{topo}
            Let $F(Z) = \sum_{i=1}^{n}Z_{i}(a_{i},b_{i},c_{i},d_{i})$ be a mixed Pham-Brieskorn polynomial, where $a_{i}, b_{i}, c_{i}, d_{i}$ satisfy the conditions in \eqref{cnd} with the additional hypothesis that $a_{i}-b_{i}, c_{i}-d_{i} > 0$ for all $i=1, \dots, n$. Then there exists a homomorphism (for the Euclidean topology) $\phi : (\bc^{*})^{n} \longrightarrow (\bc^{*})^{n}$ such that $F \circ \phi (Z) = G(Z)$, where $G(Z) = \sum_{i=1}^{n}Z_{i}(a_{i}-b_{i},0,c_{i}-d_{i},0)$.
        \end{Prop}
            \begin{proof}   
                Let $D = \{ z_{11}z_{21}\dots z_{1n}z_{2n} = 0\} \subset \co^{2n}$ and define a coordinate change $\phi : \co^{2n}\setminus D \longrightarrow \co^{2n}\setminus D$ given by $\phi(z_{1},z_{2}) = w_{1}\ebf + w_{2}\ebft$ in $\co^{2n}$ as follows. For each $i$, we take
                \begin{equation}\label{eqt}
                \begin{aligned}
                    w_{1i} &= z_{1i}\nm z_{1i}\nm^{k_{1}}\nm z_{2i}\nm^{k_{2}}, \\
                    w_{2i} &= z_{2i}\nm z_{1i}\nm^{k_{1}}\nm z_{2i}\nm^{k_{2}},
                \end{aligned}
                \end{equation}
                where $k_{1}$ and $k_{2}$ are positive rational numbers that satisfy
                \begin{align*}
                    k_{1}\left((a_{i}-b_{i}) +(c_{i}-d_{i})\right) &= 2b_{i}, \\
                    k_{2}\left((a_{i}-b_{i}) +(c_{i}-d_{i})\right) &= 2d_{i}.
                \end{align*}
                A direct computation shows that
                \begin{align*}
                    \nm z_{1i} \nm = \frac{ \nm w_{2i} \nm^{\frac{k_{1}+1}{k_{1}+k_{2}+1}}}{\nm w_{1i} \nm^{\frac{k_{1}}{k_{1}+k_{2}+1}}},\\
                    \nm z_{2i} \nm = \frac{ \nm w_{1i} \nm^{\frac{k_{2}+1}{k_{1}+k_{2}+1}}}{\nm w_{2i} \nm^{\frac{k_{2}}{k_{1}+k_{2}+1}}}.
                \end{align*}
                It follows that we may isolate $z_{1i}$ and $z_{2i}$ in \eqref{eqt} and $\phi$ is an invertible map. Finally, we obtain 
                \begin{align*}
                    w_{1i}^{a_{i}-b_{i}}w_{2i}^{c_{i}-d_{i}} &= z_{1i}^{a_{i}}\bar{z}_{1i}^{b_{i}}z_{2i}^{c_{i}}\bar{z}_{2i}^{d_{i}}, \\
                    w_{1i}^{c_{i}-d_{i}}w_{2i}^{a_{i}-b_{i}} &= z_{1i}^{c_{i}}\bar{z}_{1i}^{d_{i}}z_{2i}^{a_{i}}\bar{z}_{2i}^{b_{i}}.
                \end{align*}
                as desired.
            \end{proof}

\section{Fibrations}

For complex polar weighted homogeneous polynomials, the results in \cite{Cisneros2008} show that the zero is an isolated critical value. This implies later that one has an associated locally trivial fibration. This is generalized to bicomplex polynomials by considering the set of zero divisors. Additionally, we show the existence of spherical fibrations with base spaces being the quadric $\cc$ and $\mathbb{S}^{3}_{0}$. Henceforth, we shall denote $V_{F} = F^{-1}(\zd)$. First, we present an example.

\begin{Exam}
    \normalfont
    Let $F : \bc^{2} \longrightarrow \bc$ be given by $F(Z,W) = Z^{2} + W^{2}$. Then $F$ is weighted homogeneous of type $(1,1;2)$. If we set $Z  = \lambda_{1} + \ju \lambda_{2}$ and $W = \alpha_{1} + \ju \alpha_{2}$, where $\lambda_{i},\alpha_{i} \in \co$, for $i=1,2$, then the corresponding holomorphic map $F : \co^{4} \longrightarrow \co^{2}$ is given by:
    \begin{align*}
        F(\lambda_{1},\lambda_{2},\alpha_{1},\alpha_{2}) = (\lambda_{1}^{2} - \lambda_{2}^{2} + \alpha_{1}^{2} - \alpha_{2}^{2}, 2\lambda_{1}\lambda_{2} + 2\alpha_{1}\alpha_{2}).
    \end{align*}
    If $P = (\lambda_{1}, \lambda_{2}, \alpha_{1}, \alpha_{2}) \in \co^{4}$ is a critical point, then $z_{2} = \pm \iu z_{1}$ and $\alpha_{2} = \pm \iu \alpha_{1}$. That is, $\Sigma_{F} = \zd \times \zd$. Notice that $F(\zd \times \zd) = 0 \in \zd$. We shall see that the fact of the image of critical points being a subset of zero divisors is a general property for polar weighted polynomials.
\end{Exam}

\begin{Prop}[Proposition 3.2, \cite{Cisneros2008}]\label{regvl}
    Let $F$ be a bicomplex polar weighted homogeneous polynomial. Then every $U \in \bc^{*}$ is a regular value. 
\end{Prop}
    \begin{proof}
        Let $Z \in \bc^{n}$ such that $F(Z) = U \in \bc^{*}$. We shall see that the orbit of the action has four linearly independent real tangent vectors at $Z$. Moreover, their image by $DF_{Z}$ has real dimension 4 and so the rank of $DF$ is maximal at such a point $Z$. We may write $F = \nm F \nm e^{\iu\theta_{F}}e^{\ju\Theta_{F}}$, provided that $U$ is not a zero divisor neither zero. We have that: 
        \begin{align*}
            V_{r} &= \frac{d}{dr}\left(r \cdot Z\right)_{r=1} = \left( t_{1}r^{t_{1}-1}Z_{1}, \dots, t_{n}r^{t_{n}-1}Z_{n}\right)_{r = 1} = \left(t_{1}Z_{1}, \dots, t_{n}Z_{n}\right), \\
            V_{\theta} &= \frac{d}{d\theta}\left( e^{\iu\theta} \cdot Z\right)_{\theta = 0} = \left(\iu p_{1}e^{\iu p_{1}\theta}Z_{1}, \dots, \iu p_{n}e^{\iu p_{n}\theta}Z_{n}\right)_{\theta = 0} = \iu\left(p_{1}Z_{1}, \dots, p_{n}Z_{n}\right), \\
            V_{\Theta} &= \frac{d}{d\Theta}\left( e^{\ju\Theta} \cdot Z\right)_{\Theta = 0} = \left(\ju u_{1}e^{\ju u_{1}\Theta}Z_{1}, \dots, \ju u_{n}e^{\ju u_{n}\Theta}Z_{n}\right)_{\Theta = 0} = \ju\left(u_{1}Z_{1}, \dots, u_{n}Z_{n}\right).           
        \end{align*}
        Let us denote $P = (1,0,0)$ the unit element $1$ in the polar form and assume, without lost of generality, that $d > 0$. Then:
        \begin{align*}
            dF_{Z}\left( V_{r}\right) &= \frac{d}{dr}\left[ F (r \cdot (Z,\cjt{Z}, \cjh{Z},\cj{Z}))\right]_{|P = (1,0,0)} = \frac{d}{dr}\left( r^{a}\nm F \nm, \theta_{F}, \Theta_{F}\right)_{P=(1,0,0)} = a\nm U \nm\frac{\partial}{\partial r}, \\
            dF_{Z}\left( V_{\theta}\right) &= \frac{d}{d\theta}\left[F (\theta \cdot (Z,\cjt{Z}, \cjh{Z},\cj{Z}))\right]_{|P = (1,0,0)} = \frac{d}{d\theta}\left( \nm F \nm, \theta_{F}+c\theta, \Theta_{F}\right)_{P=(1,0,0)} = c\frac{\partial}{\partial \theta}, \\
            dF_{Z}\left( V_{\Theta}\right) &= \frac{d}{d\Theta}\left[F (\Theta \cdot (Z,\cjt{Z}, \cjh{Z},\cj{Z}))\right]_{|P = (1,0,0)} = \frac{d}{d\Theta}\left( \nm F \nm, \theta_{F}, \Theta_{F} + d\Theta + d'\overline{\Theta}\right)_{P=(1,0,0)} = d\frac{\partial}{\partial \Theta}.
        \end{align*}
       The three vectors generate a complex plane, or a real space with dimension $4$. Therefore, we obtain the existence of 4 linearly independent vectors in the image of $dF_{Z}$. 
       \end{proof}

       \begin{Not}
           \normalfont
           Henceforth we use the following convention. We always write $d$ or $d'$ if these are not zero. Otherwise, the associated exponentials $e^{\frac{\ju\Theta}{d}}$ or $e^{\frac{\ju\overline{\Theta}}{d}}$ are read as constant equal to $1$. 
       \end{Not}
    
\begin{Prop}\label{fib1}
    The map $F: \bc^{n} \setminus V_{F} \longrightarrow \bc^{*}$ is a locally trivial fibration. 
\end{Prop}
        \begin{proof}
            Define an open cover of $\bc^{*}$ given by:
            \begin{align*}
               \mathcal{U}_{0} &= \{ Z \in \bc^{*} : \Re\left(\arg_{\iu} Z\right) \neq 0\}, \\
               \mathcal{U}_{\pi} &= \{ Z \in \bc^{*} : \Re\left(\arg_{\iu} Z\right) \neq \pi\}.
            \end{align*}
            The following maps and their inverses give the local trivialization:
            \begin{multicols}{2}
            \noindent
            \begin{align*}
                \tau_{0} : F^{-1}(1) \times \mathcal{U}_{0} & \longrightarrow F^{-1}(\mathcal{U}_{0}) \\
                           (Z, se^{\iu\theta}e^{\ju\Theta}) & \longrightarrow s^{\frac{1}{a}}e^{\iu\frac{\theta}{c}}e^{\ju\frac{\Theta}{d}}e^{\ju\frac{\overline{\Theta}}{d'}}\cdot Z 
            \end{align*}
            \begin{align*}
                \tau_{0}^{-1} : F^{-1}(\mathcal{U}_{0}) & \longrightarrow F^{-1}(1) \times \mathcal{U}_{0} \\
                            W &\longmapsto \left( \rho(W)\cdot W, F(W)\right),
            \end{align*}
            \end{multicols}
            \noindent where
            $$\rho(W) = \nm F(W) \nm^{-\frac{1}{a}}e^{-\iu\frac{\arg(\nm F(W)\nm_{\iu})}{c}}e^{-\ju\frac{\arg_{\iu}(F(W))}{d}}e^{-\ju\frac{\overline{\arg_{\iu}(F(W))}}{d'}}.$$ 
            Analogously for $\tau_{\pi}: F^{-1}(1) \times \mathcal{U}_{\pi} \longrightarrow F^{-1}(\mathcal{U}_{\pi})$ and its inverse, both defined by the same expressions as above.
        \end{proof}

The polar representation of an invertible element allows us to follow the same lines of \cite[Proposition 3.4]{Cisneros2008} and conclude the existence of a fibration over $\cc$. On the other hand, the characterization of zero divisors, the radial action, and the idempotent representations yield sufficient properties to ensure a spherical fibration with base space $\mathbb{S}^{3}_{0} \subset \bc \simeq \re^{4}$.

\begin{Lem}\label{lemtrans}
    Let $F :\bc^{n}\setminus V_{F} \longrightarrow \bc^{*}$ be a bicomplex polar weighted homogeneous polynomial. Consider the projections $\pi$ and $\pi_{\iu}$ defined on $\bc^{*}$. Then the fibers of $\pi \circ F$ and $\pi_{\iu} \circ F$ are transveral to any sphere $\mathbb{S}^{4n-1}_{\epsilon}$, where $\epsilon > 0$.
\end{Lem}
    \begin{proof}
        The radial action has orbits that are transversal to any sphere $\mathbb{S}_{\epsilon}^{4n-1}$ (see, for instance, \cite[Proposition 3.4]{Cisneros2008} or \cite[Lemma 2.9]{Ruas2002}). Moreover, for every $\eta \in \mathbb{S}^{3}_{0}$ or $\cc$, the respective fibers $F^{-1}(\pi^{-1}(\eta))$ and $F^{-1}(\pi_{\iu}^{-1}(\eta))$ contain the $\re^{+}$-orbit of any point $Z \in \bc^{*}$ whose image is $\eta$.
    \end{proof}

\begin{The}\label{sfib1}
   The map
   $$\phi(Z) := \frac{F(Z)}{\nm F(Z) \nm_{\iu}} : \mathbb{S}^{4n-1}_{\epsilon}\setminus K_{\epsilon} \longrightarrow \cc$$
   is a locally trivial fibration for any $\epsilon > 0$.
\end{The}
    \begin{proof}     
        By Lemma \ref{lemtrans}, it is enough to construct the trivialization maps. Let $e^{\ju\Theta} \in \cc$ and $\delta > 0$ sufficiently small. Set $\mathcal{U}_{\delta} = \{ e^{\ju(z+\Theta)} \in \cc: \nm z \nm < \delta\}$ and consider:
            \begin{align*}
                \tau : \mathcal{U}_{\delta} \times \phi^{-1}(e^{\ju\Theta}) & \longrightarrow \phi^{-1}(\mathcal{U}_{\delta}) \\ 
                (e^{\ju(z +\Theta)}, Z) &\longmapsto e^{\ju\frac{z}{d}}e^{\ju\frac{\overline{z}}{d'}}\cdot Z \\
                \tau^{-1} : \phi^{-1}(\mathcal{U}_{\delta})& \longrightarrow \mathcal{U}_{\delta} \times \phi^{-1}(e^{\ju\Theta}) \\
                W & \longmapsto \left(e^{\ju\arg_{\iu}\phi(W)}, e^{\ju\left( \frac{-\arg_{\iu}\phi(W) + \Theta}{d}\right)}e^{\ju\left( \frac{-\overline{\arg_{\iu}\phi(W) + \Theta}}{d'}\right)}\cdot W\right).
            \end{align*}
    \end{proof}

Let $F$ be a bicomplex polar weighted homogeneous polynomial. Let $\epsilon, \delta > 0$ be arbitrary positive numbers and denote $\mathbb{S}^{3}_{\delta,0} := \mathbb{S}^{3}_{\delta} \setminus \zd$, where $\mathbb{S}^{3}_{\delta}$ is the usual $3$-sphere with radius $\delta$. By Lemma \ref{lemtrans} and Ehresmann fibration theorem, the map
\begin{align*}\label{tube1}
    F : \mathcal{N}(\epsilon,\delta) \longrightarrow \mathbb{S}^{3}_{\delta,0}
\end{align*}
is a locally trivial fibration, called \textit{Milnor-Lê} fibration, where $\mathcal{N}(\epsilon,\delta) = \mathbb{B}^{4n}_{\epsilon} \cap F^{-1}(\mathbb{S}^{3}_{\delta,0})$ is called the \textit{Milnor tube}. We denote its restriction to the interior $\mathring{\mathcal{N}}(\epsilon,\delta)$ by $\mathring{F}$. The properties of $F$ will imply an equivalent spherical fibration with base space $\mathbb{S}^{3}_{0}$ that remounts to the classical Milnor fibration theorem for holomorphic functions and generalize \cite[Proposition 3.4]{Cisneros2008}.

\begin{The}\label{sfib2}
    The map 
    $$\varphi(Z) := \frac{F(Z)}{\nm F(Z) \nm} : \mathbb{S}^{4n-1}_{\epsilon} \setminus K_{\epsilon} \longrightarrow  \mathbb{S}^{3}_{0}$$ 
    is a locally trivial fibration for any $\epsilon > 0$. Moreover, $\varphi$ is smoothly equivalent to the Milnor-Lê fibration $\pi_{\delta} \circ \mathring{F}$ for every $\delta > 0$, where $\pi_{\delta} : \mathbb{S}^{3}_{\delta} \longrightarrow \mathbb{S}^{3}$ is the normalization map.
\end{The}
        \begin{proof}
            Let $F : \re^{4n} \longrightarrow \re^{4}$ be the map $F$ regarded as a real polynomial map. The proof of Lemma \ref{lemtrans} shows that for any $\eta \in \bc^{*}$, the fiber $F^{-1}(\eta)$ intersects all spheres $\mathbb{S}_{\epsilon}^{4n-1}$ transversely. This guarantees the so-called $d$-regularity and transversality properties (see \cite{Cisneros2023} for details). Moreover, by Proposition \ref{p22} the discriminant of $F$ is linear. Hence the statement follows from \cite[Theorems 2.13 and 2.16]{Cisneros2023}.
        \end{proof}
Notice that the complex-valued norm of real numbers is the usual norm. This implies that the fibrations are related by the following commutative diagram:
     \[
  \begin{tikzcd}
    \mathbb{S}_{\epsilon}^{4n-1}\setminus K_{\epsilon} \arrow{r}{\varphi} \arrow[swap]{dr}{\phi} & \mathbb{S}^{3}_{0} \arrow{d}{\pi_{\iu}} \\
     & \mathbb{S}^{1}_{\co}
  \end{tikzcd}
    \]

\begin{Rem}
    \normalfont
    An alternative proof of Theorem \ref{sfib1} is to prove Theorem \ref{sfib2} first and consider its composition with $\pi_{\iu} : \mathbb{S}^{3}_{0} \longrightarrow \cc$. Then we conclude with \cite[Corollary 7]{McKay}, which states that the composition of the fiber bundles is a fiber bundle for smooth manifolds. This is a nontrivial result (see the discussion in \cite{Question}).
\end{Rem}
    
\section{Join theorem}

 In this section, we prove a join type theorem for bicomplex polar weighted homogeneous polynomials.

        \begin{The}
            Let $G : \bc^{n} \longrightarrow \bc$ and $H : \bc^{m} \longrightarrow \bc$ be bicomplex polar weighted homogeneous polynomials. Let $F : \bc^{n} \times \bc^{m} \longrightarrow \bc$ be defined by:
                \begin{align*}
                    F(Z,W) = G(Z) + H(W)
                \end{align*}
            and the fibers
            \begin{align*}
                X &= F^{-1}(1) \subset \bc^{n+m}, \\
                Y &= G^{-1}(1) \subset \bc^{n}, \\
                T &= H^{-1}(1) \subset \bc^{m}.
            \end{align*}
            Then, there is a homotopy equivalence $\alpha: X \longrightarrow Y \ast T$, where $Y \ast T$ denotes the Join product of $Y$ and $T$.
        \end{The}
            \begin{proof}
                The proof is almost identical to that in \cite[Theorem 4.1]{Cisneros2008}. We only remark on how to modify each homotopy in this new context. First, suppose that $G(Z)$ and $H(W)$ have radial, polar, and complex polar degrees $a,b,c, c'$ and $d,e,f, f'$, respectively.
                \begin{enumerate}
                \item  Defining $\cjt{X}$: We define a quotient space $\cjt{X}$ of $X$ by the following equivalences. A point $(Z,W) \sim (Z',W')$ if and only if:
                \begin{enumerate}
                    \item $Z = Z'$ and $H(W) = H(W')$, whose class is denoted by $[Z,*]$.
                    \item $W = W'$ and $G(Z) = G(Z')$, whose class is denoted by $[*,W]$.
                    \item $Z = Z'$ and $W = W'$, where $G(Z) \neq 0,1$, whose class is denote by $[Z,W]$.
                \end{enumerate}
                Points of type $(3)$ could be defined in terms of $H(W) \neq 0,1$, since $G(Z) + H(W) = 1$ for $(Z, W) \in X$. The topology of $X$ is the weakest topology such that the projections on each coordinate of $[Z,W]$ are continuous. Notice that this structure is coarser than the quotient one. 

                \item $\cjt{X}$ has the homotopy type of $X$: Let $\epsilon > 0$ and define
                \begin{align*}
                    N_{1,\epsilon} &= \{ (Z,W) \in X : \nm H(W) \nm \le \epsilon\}, \\
                    N_{2,\epsilon} &= \{ (Z,W) \in X : \nm G(Z) \nm \le \epsilon\}.
                \end{align*}
                Let $\rho(t) : \re \longrightarrow \re$ be a smooth function such that 
                \begin{enumerate}
                    \item[a)] $\rho(t)$ is monotone decreasing on $(\epsilon, 2\epsilon)$.
                    \item[b)] $\rho^{-1}(1) = (-\infty, \epsilon)$ and $\rho^{-1}(0) = [2\epsilon,\infty)$.
                \end{enumerate}
                Define $\mathcal{H} : X \times I \longrightarrow X$ by:
                \begin{enumerate}
                    \item[a)] For $(Z,W) \notin N_{1,2\epsilon} \cup N_{2,\epsilon}$, the homotopy is the identity.
                    \item[b)] For $(Z,W) \in N_{1,2\epsilon}$, $\mathcal{H}\left([Z,W],t\right) = (Z(t), W(t))$, where
                        \begin{align*}
                            Z(t) &= \left(\frac{r_{t}}{r_{0}}\right)^{\frac{1}{a}}\cdot e^{\iu\frac{\theta_{t}}{b}}\cdot e^{\ju\frac{\Theta_{t}}{c}}e^{\ju\frac{\overline{\Theta}_{t}}{c'}}, \\
                            W(t) &= \left( 1- t\rho\left(H(W)\right)\right)^{\frac{1}{d}},
                        \end{align*}
                        and
                        \begin{align*}
                            r(t) &= \nm 1 - H(W(t))\nm, \\
                            \theta_{t} &= \arg\nml \frac{1-H(W(t))}{1-H(W)}\nmr_{\iu}, \\
                            \Theta_{t} &= \arg_{\iu}\left( \frac{1-H(W(t))}{1-H(W)}\right).
                        \end{align*}
                    \item [c)] For points in $N_{2,2\epsilon}$, we have the analogous definitions interchange the roles of $G(Z)$ and $H(W)$.
                \end{enumerate}
                It follows that $\mathcal{H}$ is continuous and preserves $X$. To establish the homotopy between $X$ and $\cjt{X}$, a map $\mathcal{H}_{1} :\cjt{X} \longrightarrow X$ is defined as follows:
                \begin{align*}
                    \cjt{\mathcal{H}}_{1} \circ \pi = \mathcal{H}_{1}.
                \end{align*}
                By the properties of the classes in $\cjt{X}$ and the definition of $\mathcal{H}$, it is well-defined and continuous. 

                \item $\cjt{X}$ has the homotopy type of $R\cjt{X}$: where 
                \begin{align*}
                    R\cjt{X} = \{ [Z,W] \in \cjt{X} : G(Z) \in \re\}.
                \end{align*}
                Since $G(Z) + H(W) = 1$ we could define $R\cjt{X}$ in terms of $H(W) \in \re$. Let us introduce a notation. For $G(Z) = G_{1}(Z) + \ju G_{2}(Z)$, we denote
                \begin{align*}
                    \Re(G(Z)) &= \Re(G_{1}(Z)), \\
                    \Im(G(Z)) &= 1/\ju \left( G(Z) - \Re(G(Z))\right).
                \end{align*}
                Also, for $t \in [0,1]$, one has
                \begin{align*}
                    \mathbf{G}_{t}(Z) &= \Re(G(Z)) + \ju(1-t)\Im(G(Z)), \\
                     \mathbf{H}_{t}(W) &= \Re(H(W)) + \ju(1-t)\Im(H(W)).
                \end{align*}
                A deformation $\mathcal{F} : \cjt{X} \times I \longrightarrow \cjt{X}$ of $\cjt{X}$ to $R\cjt{X}$ is defined by $\mathcal{F}\left([Z,W],t\right) = [Z(t), W(t)]$ as:
                \begin{enumerate}
                    \item [a)] For points $[Z,*]$ or $[*,W]$, $\mathcal{F}$ is the identity.
                    \item [b)] For points $[Z,W]$, it is given by
                            \begin{align*}
                                Z(t) &= \left(\frac{r_{t}}{r_{0}}\right)^{\frac{1}{a}}\cdot e^{\iu\frac{\alpha_{t}}{b}}\cdot e^{\ju\frac{\beta_{t}}{c}}\cdot e^{\ju\frac{\overline{\beta}_{t}}{c'}}, \\
                                W(t) &= \left(\frac{s_{t}}{s_{0}}\right)^{\frac{1}{d}}\cdot e^{\iu\frac{\gamma_{t}}{e}}\cdot e^{\ju\frac{\lambda_{t}}{f}}\cdot e^{\ju\frac{\overline{\lambda}_{t}}{f'}},
                            \end{align*}
                            where
                            \begin{align*}
                                r_{t} &= \nm  \mathbf{G}_{t}(Z) \nm, \quad s_{t} = \nm \mathbf{H}_{t}(W) \nm, \\
                                \alpha_{t} &= \arg \nm \mathbf{G}_{t}(Z) / G(Z) \nm_{\iu}, \quad \gamma_{t} = \arg \nm \mathbf{H}_{t}(W) / H(W) \nm_{\iu}, \\
                                \beta_{t} &= \arg_{\iu} \left( \mathbf{G}_{t}(Z) / G(Z)\right), \quad \lambda_{t} = \arg_{\iu} \left( \mathbf{H}_{t}(W) / H(W) \right).
                            \end{align*}
                            It is clear that
                            \begin{align*}
                                G(Z(t)) = \mathbf{G}_{t}(Z) \quad \text{and} \quad H(W(t)) = \mathbf{H}_{t}(W),
                            \end{align*}
                            which shows that $\mathcal{F}$ preserves $R\cjt{X}$. Moreover, 
                            \begin{align*}
                                \mathcal{F}(Z(t), W(t)) &= G(Z(t)) + H(W(t)) \\
                                                        &= \mathbf{G}_{t}(Z) + \mathbf{H}_{t}(W) \\
                                                        &= \Re(F(Z,W)) + \ju(1-t)\Im(F(Z,W)) \\
                                                        & =1,
                            \end{align*}
                            this implies that the homotopy $\mathcal{F}$ belongs to $\cjt{X}$.
                \end{enumerate}
                            
                \item Conclusion: The space $R\cjt{X}$ is deformed to 
                            \begin{align*}
                                R^{+}\cjt{X} = \{ [Z,W] \in R\cjt{X}: G(Z), H(W) \ge 0\},
                            \end{align*}
                            by a homotopy that depends only on the radial homogeneity and thus it is defined in the same manner. This space is proved to be homeomorphic with $Y \ast T$ again by maps depending only on this property and the assertion follows. 
                \end{enumerate}
            \end{proof}

    For the next examples, we assume on $a_{i}, b_{i}, c_{i}, d_{i}$ the conditions in \eqref{cnd} and also that $a_{i}-b_{i}, c_{i} - d_{i} > 0$ for all $i=1, \dots, n$. See also \cite[\S 6]{Oka1973}.

    \begin{Exam}[Mixed Pham-Brieskorn]
        \normalfont
        Let $F(Z) = \sum_{i=1}^{n}Z_{i}(a_{i},b_{i},c_{i},d_{i})$ be a mixed Pham-Brieskorn polynomial. Then the fiber $F^{-1}(1)$, and consequently $F^{-1}(1)$, has the homotopy type of a bouquet of spheres $\bigvee_{m}\mathbb{S}^{n-1}$, where $m = \sigma_{1} \cdots \sigma_{n}$ with
        $$\sigma_{i} = (a_{i}-b_{i}+ c_{i} - d_{i})(a_{i}-b_{i}-c_{i}+d_{i})-1.$$
        Notice that this extends to bicomplex variables the same property for complex Pham-Brieskorn polynomials (see \cite{Pham1965} and \cite{Milnor1968}).
    \end{Exam}

    \begin{Exam}[Mixed Cyclic]
        \normalfont
        Let $F(Z) = Z_{1}(a_{1}, b_{1},c_{1},d_{1}) + Z_{2}(a_{2},b_{2},c_{2},d_{2})Z_{3}(a_{3},b_{3},c_{3},d_{3})$. Write $Z_{i} = z_{1i}\ebf + z_{2i}\ebft$, where $z_{1i}, z_{2i} \in \co^{*}$ and $i = 1,2,3$. Then the fiber of the second monomial is described by 
        \begin{equation*}
            \begin{cases}
                & z_{12}^{a_{2}}\bar{z}_{12}^{b_{2}}z_{22}^{c_{2}}\bar{z}_{22}^{d_{2}}z_{13}^{a_{3}}\bar{z}_{13}^{b_{3}}z_{23}^{c_{3}}\bar{z}_{23}^{d_{3}} = 1 \\
                & z_{12}^{c_{2}}\bar{z}_{12}^{d_{2}}z_{22}^{a_{2}}\bar{z}_{22}^{b_{2}}z_{13}^{c_{3}}\bar{z}_{13}^{d_{3}}z_{23}^{a_{3}}\bar{z}_{23}^{b_{3}} = 1
            \end{cases}
        \end{equation*}
        Set $m = m_{1}\cdot m_{2}$, where $m_{1} = \gcd(a_{2}-b_{2}-c_{2}+d_{2}, a_{3}-b_{3}-c_{3}+d_{3})$ and $m_{2} = \gcd(a_{2}-b_{2}+c_{2}-d_{2}, a_{3}-b_{3}+c_{3}-d_{3})$. Then
        \begin{align*}
            F^{-1}(1) \simeq \left\{\left(a_{1}-b_{1}+c_{1}-d_{1}\right)\left(a_{1}-b_{1}-c_{1}+d_{1}\right)\;\text{points}\right\} \ast \bigsqcup_{i=1}^{m}\tr.
        \end{align*}
    \end{Exam}

\section*{Acknowledgments} We thank Professor José Luis Cisneros-Molina and Professor José Antonio Arciniega Nevárez for useful conversations and for encouraging this work. The second named author was supported by the Coordenação de Aperfeiçoamento de Pessoal de Nível Superior - Brasil (CAPES) - Finance Code 001. The third named author was supported by DGAPA PAPIIT IN112424 Complex Kleinian groups and DGAPA PAPIIT IN117523 Singularidades de superficies complejas: modificaciones, resoluciones y curvas polares. The first named author was supported by UNAM Posdoctoral program (POSDOC).

\subsection*{Conflict of interest} The authors declare no conflict of interest.

\subsection*{Data Availability} Data availability is not applicable to this article.

    \nocite{*} 
    \bibliographystyle{siam}
    \bibliography{References} 

\end{document}